\documentclass[reqno,10pt]{amsart}
\usepackage[a4paper]{geometry}
\usepackage{amsfonts, amsmath, amssymb, amsthm, color}

\usepackage{fullpage}

\allowdisplaybreaks[1]

\usepackage{graphicx,xcolor}

\usepackage{verbatim} 

\usepackage{tikzsymbols}  

\usepackage[colorlinks]{hyperref}
\hypersetup{linkcolor=black,citecolor=black,filecolor=black,urlcolor=blue}

\newtheorem{thm}{Theorem}[section]
\newtheorem{lemma}[thm]{Lemma}
\newtheorem{prop}[thm]{Proposition}
\newtheorem{cor}[thm]{Corollary}

\newtheorem{rmk}[thm]{Remark}
\usepackage{soul}

\theoremstyle{definition}
 \numberwithin{equation}{section}

\newcommand{\eps}{\varepsilon}
\newcommand{\R}{{\mathbb R}}

\DeclareMathOperator{\spann}{span}

\def\sideremark#1{\ifvmode\leavevmode\fi\vadjust{\vbox to0pt{\vss
 \hbox to 0pt{\hskip\hsize\hskip1em
 \vbox{\hsize2.1cm\tiny\raggedright\pretolerance10000
  \noindent #1\hfill}\hss}\vbox to15pt{\vfil}\vss}}}%

\title[Slightly supercritical pure Neumann problems]{Existence of solutions to a slightly supercritical pure Neumann problem}
\author{Angela Pistoia, Alberto Salda\~na, and Hugo Tavares}
\date{\today}

\begin{document}

\begin{abstract}
We show the existence and multiplicity of concentrating solutions to a pure Neumann slightly supercritical problem in a ball.  This is the first existence result for this kind of problems in the supercritical regime. Since the solutions must satisfy a compatibility condition of zero average, all of them have to change sign.  Our proofs are based on a Lyapunov-Schmidt reduction argument which incorporates the zero-average condition using suitable symmetries. Our approach also guarantees the existence and multiplicity of solutions to subcritical Neumann problems in annuli. More general symmetric domains (\emph{e.g.} ellipsoids) are also discussed.
\end{abstract}
\maketitle
\section{Introduction}\label{uno}

Let $\Omega=B_1(0)$ be the unit ball in $\mathbb R^n$ ($n\geq 4$) centered at the origin, $q>0$, $q\neq 1$, and consider the pure Neumann semilinear problem given by 
\begin{equation}\label{super:intro}
-\Delta u=|u|^{q-1}u\ \hbox{in}\ \Omega,\qquad \partial_\nu u=0\ \hbox{on}\ \partial\Omega.
\end{equation}
Solutions of this problem for a general smooth domain are known to exist only for $q\leq p$, where 
\begin{align*}
p:=\frac{n+2}{n-2}
\end{align*}
is the critical Sobolev exponent. The sublinear case $q\in(0,1)$ is studied in \cite{PW15} via a minimization problem with a nonsmooth constraint, while the superlinear-subcritical case $q\in (1,p)$ can be handled with a standard Nehari approach, see \cite{ST17}. Finally, in contrast to the Dirichlet counterpart of \eqref{super:intro} (which does not have solutions for $q\geq p$ because of the Pohozaev identity), the pure Neumann problem allows the existence of smooth solutions at the critical exponent $q=p$.  For instance, one can find a least-energy solution using the dual method \cite{CK91}; this solution is classical and can also be obtained as a $C^{2,\alpha}$--limit of slightly subcritical least-energy solutions \cite{ST22}. These solutions must change sign, since, integrating \eqref{super:intro} over $\Omega$, we have that
\begin{align}\label{ave}
    \int_{\Omega}|u|^{q-1}u\, dx=\int_{\Omega}-\Delta u \, dx= \int_{\partial \Omega}\partial_\nu u\, d\sigma=0.
\end{align}
In particular, this implies that for $p=q$ there cannot be any radially symmetric solution (because this would imply the existence of a radially symmetric solution to the Dirichlet problem in the nodal component containing the origin).

In addition to the least-energy solution, one can also construct symmetric solutions with a gluing approach, see for example \cite{CK90}.  

The proofs of these results are non trivial, since one has to overcome the lack of compactness inherent of critical problems.  This difficulty is even greater in the supercritical regime $q>p$, for which no existence result for \eqref{super:intro} was previously known.

This is the question that motivates the present paper.  In particular, we wish to answer the following questions. For  $\eps>0$, does the problem
\begin{equation}\label{super}
-\Delta u=|u|^{p-1+\eps}u\ \hbox{in}\ \Omega,\qquad \partial_\nu u=0\ \hbox{on}\ \partial\Omega,
\end{equation}
 admits a solution?  Can one obtain multiple solutions?  Is it possible to have solutions in other domains which are not the unit ball?
 
 We answer all these questions affirmatively in case the domain is either a ball or an annulus. For this, we use a Lyapunov-Schmidt reduction strategy, which has several differences with respect to its implementation in the study of Dirichlet problems.  For instance, in Neumann problems  the maxima and minima of the solutions
can be located at specific points on the boundary (and not necessarily in the
interior of the domain, as in the Dirichlet case).
 This difference implies important changes in the method, since the curvature of the boundary now plays an important role and the blow-up analysis in the Neumann case leads (after a rescaling) to a limiting problem in the halfspace (see Lemma \ref{exp:lem}). 
 
The Lyapunov-Schmidt reduction strategy has been used previously to construct \emph{positive solutions} with bubbling behavior around one or more critical points of mean curvature, with positive mean curvature, to the Neumann problem 
  \begin{equation}\label{RW-pure}
-\Delta u+\mu u=u^q\ \hbox{in}\ \Sigma,\qquad \partial_\nu u=0\ \hbox{on}\ \partial\Sigma,
\end{equation}
when the parameter $\mu\to+\infty$, $q=p$ in \cite{M94,AMY95,G95,gg98,R97,w95}, and $\Sigma \subset \R^N$ is a general bounded smooth domain. A similar approach has been also used to obtain solutions to \eqref{RW-pure}
when $\mu>0$ is fixed and $q=p\pm\eps$ with $\eps$ sufficiently small, which blows-up at critical points of the mean curvature of the boundary with positive mean curvature if $\eps>0$ and negative mean curvature if $\eps<0$, see \cite{DMP05,RW04,RW05}. In all the previous cases,  the presence of the linear term  $\mu u$ allows to use as building block the single bubble.  In the \emph{pure Neumann case}, namely, with $\mu=0$, the natural constraint (1.2) forces the solutions to be sign changing. It is natural, therefore, to look for solutions which are the sum of two single bubbles with different signs concentrated at two suitable different points. We do not known if a construction with a single bubble with a small negative part can be done.  In symmetric domains, we show below that the blow-up points can be positioned at points on the boundary that are critical points of the mean curvature and that preserve a particular symmetric arrangement after a translation, see Remark \ref{domains}.  The symmetry of the domain simplifies the choice of the two concentration points and reduces the number of unknowns in the Lyapunov-Schmidt reduction.  As a side result of independent interest, we also show the existence of new (concentrating) solutions in the slightly subcritical regime in an annuli. 

 Before we state our main results, we need to introduce some notation. Let
\begin{align}\label{eq:bubble}
U_{\delta,\xi}(x):=\alpha_n \frac{\delta^{\frac{n-2}{2}}}{(\delta^2+|x-\xi|^2)^{\frac{n-2}{2}}},\qquad \text{ where } \xi\in \R^N,\ \delta>0,\quad \alpha_n:=[n(n-2)]^\frac{n-2}{4}.
\end{align}
This family of functions are called \emph{bubbles} and represent the unique positive solutions in $\mathcal{D}^{1,2}(\R^N)$ of 
$$-\Delta U=U^p\ \hbox{in}\ \mathbb R^n.$$

For every $t>1$ and $h\in L^t(\Omega)$ satisfying $\int_\Omega h =0$, we denote by 
\begin{align}\label{K:def}
K(h)=u\in W^{2,t}(\Omega)    
\end{align}
the unique solution to the (pure) Neumann problem
\begin{align}\label{N:p}
  -\Delta u =h\quad \text{in }\Omega,\qquad \partial_\nu u=0\quad\text{ on }\partial \Omega,\qquad \int_\Omega u =0.
\end{align}
Furthermore, for $\widetilde u\in W_{loc}^{2,t}(\R^n)$ such that $\int_\Omega \Delta \widetilde u=0$, we write 
\begin{align*}
P\widetilde u:=K(-\Delta \widetilde u).    
\end{align*}

Let $W_{\delta}:=U_{\delta,e_n}-U_{\delta,-e_n}$, where $e_n:=(0,\ldots, 0,1)\in \R^n$. Then $P W_{\delta}$ is the solution of  
\begin{align*}
-\Delta P W_{\delta}=-\Delta  W_\delta = U_{\delta,e_n}^p-U_{\delta,-e_n}^p\ \hbox{in}\ \Omega,\qquad \partial_\nu P W_{\delta}=0\ \hbox{on}\ \partial\Omega, \qquad
\int\limits_{\Omega} P W_{\delta}=0.
\end{align*}
Observe that $\int_\Omega (U_{\delta,e_n}^p-U_{\delta,-e_n }^p)=0$ is satisfied.   Set 
\begin{equation}\label{eq:s_eps}
s_\eps:=p+1+\eps=\frac{2n}{n-2}+ \eps
\end{equation}
and let
\begin{align*}
  H_\eps&:=\{u\in H^1(\Omega)\cap L^{s_\eps}(\Omega)\::\: u \text{ is odd in } x_n \text{ and even in } x_i \text{ for } i=1,\ldots, n-1\}\\
  					&=\left\{u\in H^1(\Omega)\cap L^{s_\eps}(\Omega)\::\: \begin{array}{l}u(x_1,\ldots,x_{n-1}, -x_n)=-u(x_1,\ldots, x_{n-1},x_n)\\ 
																												u(x_1,\ldots,-x_i,\ldots, x_n)=u(x_1,\ldots,x_i,\ldots, x_n),\ i=1,\ldots, n-1
																												 \end{array}\right\},
\end{align*}
endowed with the norm 
\begin{align*}
\|u\|_{H_\eps} = \|u\| + |u|_{s_\eps}.
\end{align*} 

Our main existence result is the following.

\begin{thm}\label{thm:main:ball} Let $n\geq 4$ and let $\Omega=B_1(0)\subset \R^n$ be the unit ball centered at the origin.  Then there is $\eps_0>0$ such that, for $\eps\in (0,\eps_0)$, the problem \eqref{super} has a solution $u_\eps\in H_\eps$ of the form
\[
u_\eps=PW_{\delta_\eps,e_n} + \phi_\eps,
\]
where 
\begin{align*}
\delta_\eps=d(\eps) \eps,    
\end{align*}
with $d(\eps)\to d_*$ as $\eps\to 0$, for $d_*>0$ given explicitly by \eqref{eq:d_n} below. Moreover,  $\phi_\eps\in H_\eps$ is such that $\|\phi_\eps\|_{H^1(\Omega)}\to 0$ as $\eps\to 0$.
\end{thm}

As mentioned before, Theorem \ref{thm:main:ball} is proved using a Lyapunov-Schmidt reduction method in the space $H_\eps$. In particular, these solutions are odd and therefore we have that
\begin{align*}
  \int_\Omega u_\eps = \int_\Omega |u_\eps|^{t}u_\eps =0\qquad \text{ for any $0<t \leq p+\eps$.} 
\end{align*}
Of course, the problem \eqref{super} could have solutions which do not belong to $H_\eps$. For instance, we do not know if a solution can be constructed using only \emph{one} bubble.

We believe that our  techniques can be adapted to other situations, which we describe next.  Theorem \ref{thm:main:ball} is concerned with solutions which look like the difference of two bubbles. In the same spirit of  \cite{PFM03} and \cite{PR06}, it is natural to guess the existence of solutions on the unit ball that look like the sum of $2k$ bubbles with alternating signs. More precisely, for $\xi=(\xi_1,\ldots,\xi_n)\in \R^n$ and $\theta\in[0,2\pi)$, let
\begin{align*}
R_\theta:\R^n\to\R^n\quad \text{ be given by }\quad R_\theta 
\xi=(\xi_1 \cos\theta-\xi_n\sin\theta,\xi_2,\ldots,\xi_{n-1},\xi_1\sin\theta+\xi_n\cos\theta),
\end{align*}
namely, $R_\theta$ is a rotation through an angle $\theta$ in the plane generated by $e_1=(1,0,\ldots,0)$ and $e_n=(0,\ldots,0,1)$. Then, for small $\eps\in (0,\eps_0)$, the supercritical problem \eqref{super} should have a solution $u_\eps\in H_\eps$ of the form
\begin{align}\label{form}
u_\eps=P \left(\sum_{i=0}^{2k-1} (-1)^i U_{\delta^\eps,\xi_i}\right) + \phi_\eps,  
\end{align}
where $\xi_i:=R_{i\frac{\pi}{k}}e_n$, $\delta_{\eps}=d(\eps) \eps$, $d(\eps)\to d^*$ as $\eps\to 0$, and $\phi_\eps\in H_\eps$ is such that $\|\phi_\eps\|_{H^1(\Omega)}\to 0$ as $\eps\to 0$. In Remark \ref{domains} we explain why the blow-up points $x_i$ require these particular arrangement.  Comparing with Theorem \ref{thm:main:ball}, the proof of this result should be more technical and we leave it as an open problem.  It would be also interesting to investigate the existence of solutions exhibiting a clustering phenomena as in  \cite{WY07}.

We also mention that our approach should be easily adapted to other symmetric domains.  For instance, in an ellipsoid in $\R^n$ ($n\geq 4$) one can center two bubbles with opposite signs at the two antipodal points with the largest mean curvature (the vertex) and also other two bubbles with opposite signs at the two antipodal points with the smallest mean curvature (the co-vertex). In fact, the $n-$dimensional ellipse may have  $n$ different solutions blowing up at the critical points of the mean curvature which lie on opposite axis. A similar approach can be used to consider symmetric bounded smooth domains (not necessarily convex).  For instance, let
    \begin{equation}\label{eq:assumption_Omega}
\Omega \text{ be symmetric with respect to $x_i$ for every $i=1,\ldots, n$,}
\end{equation}
that is, $(x_1,\ldots, x_{i-1},-x_i,x_{i+1},\ldots, x_n)\in \Omega \iff (x_1,\ldots, x_{i-1},x_i,x_{i+1},\ldots, x_n)\in \Omega$.
Furthermore, given $k\in\mathbb N$ and $R_\theta$ as before, assume that $\Omega$ is invariant under the rotation $R_{\frac{\pi}{k}}$.  If $\xi\in \partial \Omega$ is a critical point of the mean curvature with \emph{positive mean curvature}, then \eqref{form} holds with $\xi_i=R_{i\frac{\pi}{k}}$ with $i=0,\ldots,2k-1.$
We refer the reader to Remarks \ref{domains} and \ref{rmk:phi0_generalOmega} for more details.  This shows an important difference between our approach and the one in \cite{RW05}, where the concentration point must be at a point maximizing the main curvature. In our case, the symmetries imposed on the domain allow for different configurations.

In the case of a general $C^1$ domain (without symmetries), however, it remains  an open question the existence of solutions to \eqref{super}.  The main difficulty is that, without the symmetries, it is not clear where to position the concentration points in order to maintain the zero average constraint needed to solve Neumann problems.

Finally, we mention that the techniques presented in this paper can also be used to guarantee the existence of blowing-up solutions to slightly subcritical problems in symmetric domains, but in the subcritical case the blow-up points must be positioned at points of \emph{negative curvature}.  For example, we have the following. 

\begin{thm}\label{thm:main:sub} Let $n\geq 4$, $k\in\mathbb N$, and let $\Omega\subset \R^n$ be a an annulus centered at the origin
\begin{align*}
\Omega:=\{x\in\R^n\::\:  a<|x|<b\} \qquad \text{ for some }0<a<b.
\end{align*}
There exists $\eps_0>0$ such that, for $\eps\in (0,\eps_0)$, there is a solution $u_\eps\in H_\eps$ of the form \eqref{form} for
\begin{enumerate}
    \item the subcritical problem
\begin{equation*}
-\Delta u=|u|^{p-1-\eps}u\ \hbox{in}\ \Omega,\quad \partial_\nu u=0\ \hbox{on}\ \partial\Omega,
\end{equation*}
where $\xi_i:=a R_{i\frac{\pi}{k}}e_n$, $\delta^{\eps}=d^{\eps} \eps$, $d^{\eps}\to d^*$ as $\eps\to 0$, and $\phi_\eps\in H_\eps$ is such that $\|\phi_\eps\|_{H^1(\Omega)}\to 0$ as $\eps\to 0$.
\item the supercritical problem
\begin{equation*}
-\Delta u=|u|^{p-1+\eps}u\ \hbox{in}\ \Omega,\quad \partial_\nu u=0\ \hbox{on}\ \partial\Omega,
\end{equation*}
where $\xi_i:=b R_{i\frac{\pi}{k}}e_n$, $\delta^{\eps}=d^{\eps} \eps$, $d^{\eps}\to d^*$ as $\eps\to 0$, and $\phi_\eps\in H_\eps$ is such that $\|\phi_\eps\|_{H^1(\Omega)}\to 0$ as $\eps\to 0$.
\end{enumerate}

\end{thm}
These solutions would be the pure Neumann analog of the positive solutions of \eqref{RW-pure} found in \cite{RW05}.

\medskip

To close this introduction, we point out some closing remarks. In this work, we have only considered dimensions $n\geq 4$.  Dimension $n=3$ requires more delicate computations (cf. \cite{R97} for the correction term in a similar problem) and we do not pursue this here. Furthermore, here we have only considered smooth domains, but we believe a similar approach could also be used in symmetric domains with corners such as a cube $\{x=(x_1,\ldots,x_n)\in \R^n\::\:|x_i|<1\text{ for }i=1,\ldots,n\}$ with concentrating points $e_n$ and $-e_n$.  Other polygonal domains could also be considered. If the concentration points are placed at the corners, we point out that the expansion given in Lemma \ref{exp:lem} would need to be adjusted with a different limiting profile for $\varphi_0$. 

The paper is organized as follows.  In Section \ref{Sec:prel} we give some preliminaries and construct the Ansatz that we use in our proofs.  In Section \ref{sec:lin} we reduce the problem of finding a solution of \eqref{super} to finding a critical point of a functional in a space of dimension one. Finally, in Section \ref{sec:red} we show that the reduced problem does have a critical point and in Section \ref{op:sec} we discuss some open problem.

\subsection{Acknowledgments}
We thank the referees for their careful reading of our paper and for their valuable comments and suggestions that helped us to substantially improve this paper.  A. Salda\~{n}a is supported by UNAM-DGAPA-PAPIIT grants IA101721 and IA100923 (Mexico), by CONACYT grant A1-S-10457 (Mexico), and by the 2021 Visiting Professor Programme of La Sapienza University (Italy). H. Tavares is partially supported by the Portuguese government through FCT-Funda\c c\~ao para a Ci\^encia e a Tecnologia, I.P., under the projects UID/MAT/04459/2020 and PTDC/MAT-PUR/1788/2020.

\section{Preliminaries and the Ansatz.}\label{Sec:prel}

Recall that, from now on, we take $\Omega=B:=B_1(0)\subset \R^n$, $n\geq 4$ and $p=(n+2)/(n-2)$.  For $t>0$, let
\begin{align*}
|u|_t:=\left(\int_{{B}}|u|^t\right)^\frac{1}{t}\qquad \langle u,v\rangle:=\int_{B} \nabla u \cdot \nabla v\qquad \text{ and }\qquad \|u\|:=\left(\int_{B} |\nabla u|^2\right)^\frac{1}{2}.
\end{align*}
In particular, $\|\cdot\|$ is an equivalent norm in the Hilbert space $\{u\in H^1({B}):\ \int_{B} u=0\}$.

It is well known (see \cite{BianchiEgnell}) that the space of solutions of the linearized equation
\begin{equation}
-\Delta V=pU_{\delta,\xi}^{p-1}V,\qquad V\in \mathcal{D}^{1,2}(\R^n)
\end{equation}
has dimension $n+1$, being spanned by 
\[
\partial_\delta U_{\delta,\xi}(x)= \alpha_n\frac{n-2}{2}\delta^\frac{n-4}{2}\frac{|x-\xi|^2-\delta^2}{(\delta^2+|x-\xi|^2)^\frac{n}{2}},\quad \partial_{\xi_i} U_{\delta,\xi}(x)=-\frac{\alpha_n\delta^\frac{n-2}{2}(x_i-\xi_i)}{(\delta^2+|x-\xi|^2)^\frac{n}{2}},\ i=1,\ldots, n,
\] 
where $U_{\delta,\xi}$ is given by \eqref{eq:bubble}.

Therefore, the space of solutions of
\begin{equation}\label{eq:linearized_eq}
-\Delta v=p U_{1,0}^{p-1} v,\qquad v\in \mathcal{D}^{1,2}(\R^n),
\end{equation}
which are even in $x_1,\ldots,x_{n-1}$,
\begin{align}\label{Vdef}
\text{ is spanned by }V:=\partial_\delta U_{\delta,0}\big|_{\delta=1}=\alpha_n\frac{n-2}{2}\frac{|x|^2-1}{(1+|x|^2)^\frac{n}{2}}.
\end{align}

For future convenience, we observe that
\begin{equation}\label{pa U con U}
\delta\left|\partial_\delta U_{\delta,\xi}(x)\right|  \le  U_{\delta,\xi}(x).
\end{equation}

\begin{lemma}\label{lemma:Kcontinuous} The operator $K$ (given in \eqref{K:def}) satisfies the following. 
\begin{enumerate}
\item Let $h\in L^\frac{2n}{n+2}({B})$ with $\int_{B} h=0$. Then there exists $c>0$ such that
\begin{equation}\label{eq:Kinequality}
\|Kh\| \leq c|h|_\frac{2n}{n+2}.
\end{equation}
\item Let $s>\frac{n}{n-2}$ and $h\in L^{\frac{ns}{n+2s}}({B})$ with $\int_{B} h=0$. Then $Kh\in L^s({B})$ and
\begin{equation}\label{eq:Kcontinuous}
    |Kh|_s\leq c|h|_{\frac{ns}{n+2s}}
\end{equation}
for some positive constant $c$ which depends only on $n$, $s$ and ${B}$.
\end{enumerate}
\end{lemma}
\begin{proof} Proof of (a). By using integration by parts, H\"older's inequality, and Sobolev embeddings, 
\[
\|Kh\|^2=\int_{B} \nabla (Kh)\cdot \nabla (Kh)=\int_{B} (Kh) h \leq |Kh|_\frac{2n}{n-2} |h|_\frac{2n}{n+2}\leq c\|Kh\|\, |h|_\frac{2n}{n-2},
\]
from which \eqref{eq:Kinequality} follows.

Proof of (b). The assumptions imply that $\frac{ns}{n+2s}>1$, so that $L^{\frac{ns}{n+2s}}({B})\subset L^1({B})$ and $\int_{B} h$ is well defined. By elliptic regularity theory (see \cite[Theorem and Lemma in page 143]{RR85} or \cite[Theorem 15.2]{ADN59}) there exists $C=C({B},n,s)$ such that $\|Kh\|_{2,\frac{ns}{n+2s}}\leq C |h|_\frac{ns}{n+2s}$, and \eqref{eq:Kcontinuous} follows from the Sobolev embedding $W^{2,\frac{ns}{n+2s}}({B})\hookrightarrow L^s({B})$. 
\end{proof}

We also use the following notation:
\begin{align}\label{omegas}
    \omega_+:={B}\cap B_{\frac{1}{2}}(e_n)\qquad \text{ and }\qquad 
    \omega_-:={B}\cap B_{\frac{1}{2}}(-e_n).
\end{align}

As in the previous section, we let $W_\delta:=U_{\delta,e_n}-U_{\delta,-e_n}$ and let $P W_{\delta}$ be the solution of 
\begin{align*}
-\Delta PW_\delta=-\Delta  W_\delta = U_{\delta,e_n}^p-U_{\delta,-e_n}^p\ \hbox{in}\ {B},\qquad
 \partial_\nu P W_{\delta}=0\ \hbox{on}\ \partial{B}, \qquad
\int\limits_{{B}} P W_{\delta}=0.
\end{align*}
That is, $PW_\delta=K(-\Delta PW_\delta)=K(U_{\delta,e_n}^p-U_{\delta,-e_n}^p)$. Observe that, since $U_{\delta,e_n}^p-U_{\delta,-e_n}^p$ is odd  in $x_n$ and even in $x_1,\ldots, x_{n-1}$, then so is $PW_\delta$.
Let $\R^n_+:=\{x\in \R^n\::\: x_n>0\}$ and let $\varphi_0$ be the solution of
\begin{align}\label{phieq}
-\Delta \varphi_0 = 0 \text{ in }\R^n_+,\qquad
  \frac{\partial \varphi_{0}}{\partial x_n}
  =\alpha_n\frac{n-2}{2}\frac{|x'|^2}{(1+|x'|^2)^\frac{n}{2}} \text{ on }\partial \R^n_+,\qquad
 \varphi_0\to 0  \text{ as }|x|\to \infty.
\end{align}
We then have the following expansion, whose proof can be found in Appendix \ref{app:expansion}.

\begin{lemma}\label{exp:lem} 
 For $n\geq 4$, it holds that, for $x\in{B}$,
 \begin{align*}
  PW_{\delta}(x)&=W_{\delta}(x)-\delta^{-\frac{n-4}{2}}\Big(\varphi_0\Big(\frac{e_n-x}{\delta}\Big)-\varphi_0\Big(\frac{e_n+x}{\delta}\Big)\Big)+\zeta_\delta(x),
\end{align*}
  where
  \begin{align*}
  \zeta_\delta&= O(\delta^{\frac{6-n}{2}})\ \text{ and }\ \partial_\delta \zeta_\delta = O(\delta^{\frac{4-n}{2}})\quad \text{ as $\delta\to 0$ for $n\geq 5$,}\\
  \zeta_\delta &= O(\delta\log \delta)\ \text{ and }\ \partial_\delta \zeta_\delta = O(\log \delta)\quad \text{ as $\delta\to 0$ for every $\eps\in(0,1)$ and $n=4$,}
  \end{align*}
  uniformly in ${B}$. Moreover, there exists $C>0$ such that 
  \begin{align}
|\zeta_\delta(x)|,\,|PW_{\delta}(x)-W_{\delta}(x)|&\leq \frac{C\delta^\frac{n-2}{2}}{(\delta+|x-e_n|)^{n-3}}+\frac{C\delta^\frac{n-2}{2}}{(\delta+|x+e_n|)^{n-3}}\notag\\ &=\frac{C\delta^\frac{4-n}{2}}{(1+|\frac{x-e_n}{\delta}|)^{n-3}} +\frac{C\delta^\frac{4-n}{2}}{(1+|\frac{x+e_n}{\delta}|)^{n-3}}.\label{eq:estimate_PW-W}\\
|\partial_\delta \zeta_\delta(x)|,\,|\partial_\delta(PW_{\delta}(x)-W_{\delta}(x))|&\leq \frac{C\delta^\frac{n-4}{2}}{(\delta+|x-e_n|)^{n-3}}+\frac{C\delta^\frac{n-4}{2}}{(\delta+|x+e_n|)^{n-3}}\notag\\ &=\frac{C\delta^\frac{2-n}{2}}{(1+|\frac{x-e_n}{\delta}|)^{n-3}} +\frac{C\delta^\frac{2-n}{2}}{(1+|\frac{x+e_n}{\delta}|)^{n-3}}.\label{eq:estimate_derivative_PW-W}
  \end{align}
\end{lemma}

A straightforward consequence of the estimates \eqref{eq:estimate_PW-W} and \eqref{eq:estimate_derivative_PW-W} is the following.

\begin{cor}\label{cor:estimates_with_bubbles}
 For every $\tau \in\left[ \frac{n-3}{n-2},1\right]$, there exists $\kappa=\kappa(\tau)>0$ such that
\[
| PW_{\delta}-W_{\delta }|,\ | \delta \partial_\delta(PW_{\delta}-W_{\delta })| \leq \kappa (U_{\delta,e_n}+U_{\delta,-e_n})^{\tau} \qquad \text{ for every } x\in {B},
\]  
and there exists $C>0$ such that $|PW_{\delta}|,\ |\delta \partial_\delta PW_{\delta}|\leq C(U_{\delta,e_n}+U_{\delta,-e_n})$.
\end{cor}

By performing the rescaling 
\begin{align*}
y=\frac{x+e_n}{\delta}\qquad \text{for $x\in \omega^-={B}\cap B_{\frac{1}{2}}(-e_n)$}    
\end{align*}
 and 
 \begin{align*}
 y=\frac{e_n-x}{\delta}\qquad \text{for $x\in \omega^+= {B} \cap B_1(e_n)$,}
 \end{align*}
  we have that the rescaled projections have remainder terms that go to zero.
\begin{cor}\label{cor:resc_asympt}
As $\delta\to 0$, we have
\begin{align}
  \delta^{\frac{n-2}{2}}PW_{\delta}(\delta y-e_n) &= -U_{1,0}(y)+\delta \varphi_0(y)+O(\delta^{2}), \label{eq:resc_asympt1}\\
  \delta^{\frac{n-2}{2}}PW_{\delta}(\delta y+e_n) &= U_{1,0}(x)-\delta \varphi_0(-x)+O(\delta^{2}), \label{eq:resc_asympt2}
   \end{align}
uniformly in $y\in B_\frac{1}{\delta}(0)\cap B_\frac{1}{\delta}(\frac{e_n}{\delta})$.
\end{cor}
\begin{proof}
Let $x\in \omega^-$ and let $x=\delta y-e_n$. Then $y\in B_\frac{1}{2\delta}(0)\cap B_\frac{1}{\delta}(\frac{e_n}{\delta})$, $|y-\frac{2e_n}{\delta}|\geq \frac{3}{2\delta}$ and
\begin{align*}
  \delta^{\frac{n-2}{2}}PW_{\delta}(\delta y-e_n) =U_{1,-\frac{2e_n}{\delta}}(y)-U_{1,0}(y)+\delta \varphi_0(y)-\delta \varphi_0\Big(\frac{2e_n}{\delta}-y\Big)+O(\delta^{2}).
 \end{align*}
Then
\begin{align*}
U_{1,-\frac{2e_n}{\delta}}(y)&=\frac{\alpha_n}{\left(1+\left|y-\frac{2e_n}{\delta}\right|^2\right)^\frac{n-2}{2}}=O(\delta^{n-2}),\\
\delta \varphi_0(y)-\delta \varphi_0\Big(\frac{2e_n}{\delta}-y\Big) &=O\left(\frac{\delta}{\left(1+\left|y-\frac{2e_n}{\delta}\right|\right)^{n-3}}\right)=O(\delta^{n-2}),
\end{align*}
where we used the estimate \eqref{A6} in the second identity. Since $n\geq 4$, then $O(\delta^{n-2})=O(\delta^2)$ and \eqref{eq:resc_asympt1} holds true. The estimate \eqref{eq:resc_asympt2} follows in an analogous way.
\end{proof}

\subsection{The Ansatz}

Let $w_{\delta}=P W_{\delta}(x)$ and $f_\eps(t):=|t|^{p-1+\eps} t$. We search for a solution of \eqref{super} of the form
\[
u=PW_\delta+ \phi=w_\delta+\phi,
\]
where 
\[
\delta=d\eps,
\]
$d\in \{c\in \R:\ \eta<c<1/\eta\}$ for a small $0<\eta<1$, $\phi\in H_\eps$. Using this Ansatz in the equation, we obtain
\begin{align}\label{w}
  w_\delta+\phi = K(f_\eps(w_\delta+\phi)),
\end{align}
or, equivalently,
\begin{align*}
\phi-K(f'_\eps(w_\delta)\phi)=K(f_\eps(w_\delta+\phi)-f_\eps(w_\delta)-f'_\eps(w_\delta)\phi)+(K(f_\eps(w_\delta))-w_\delta).
\end{align*}
Observe that $f'_\eps(w_\delta)\phi$, $f_\eps(w_\delta+\phi)-f_\eps(w_\delta)-f'_\eps(w_\delta)\phi$ and $f_\eps(w_\delta)$ are odd in $x_n$ and even in $x_1,\ldots, x_{n-1}$, so $\phi-K(f'_\eps(w_\delta)\phi)$, $K(f_\eps(w_\delta+\phi)-f_\eps(w_\delta)-f'_\eps(w_\delta)\phi)$ and $K(f_\eps(w_\delta))-w_\delta$ have the same symmetries. The fact that they belong to $H^1({B})\cap L^{s_\eps}({B})$, and hence to $H_\eps$, will be checked in the next section.
Define 
\[
\Theta_{d,\eps}:=\spann \left\{ \delta \partial_\delta w_\delta\right\} \quad \text{ and } \quad  \Theta_{d,\eps}^\perp=\left\{\phi\in H_\eps:\ \langle \phi, \delta  \partial_\delta w_\delta\rangle=0  \right\}
\] 
and let $\Pi_\eps: H_\eps\to \Theta_{d,\eps}$, $\Pi_\eps:H_\eps\to \Theta_{d,\eps}^\perp$ be the orthogonal projections:
\[
\Pi_\eps(\phi):=\|\partial_\delta w_\delta \|^{-2}\left\langle \phi,  \partial_\delta w_\delta \right\rangle \partial_\delta w_\delta, \qquad \Pi_\eps^\perp(\phi):=\phi-\Pi_\eps(\phi).
\]

The following estimate will be used in the next section.
\begin{lemma}\label{lemma:projection_continuous} Let $0<\eta<1$ and take $\delta=d\eps$ with $\eta<d<\frac{1}{\eta}$. There exists $C=C(n,\eta)>0$ such that, as $\eps\to 0$,
\[
\|\Pi_\eps^\perp \phi\|_{H_\eps} \leq C(\|\phi\| + |\phi|_{s_\eps}) \text{ for every $\phi\in H_\eps$,}
\]
where $s_{\eps}$ is defined in \eqref{eq:s_eps}.
\end{lemma}
\begin{proof} Since $\Pi_\eps^\perp$ is a projection, it is clear that $\|\Pi_\eps^\perp \phi\|\leq \|\phi\|$ for every $\phi\in H_\eps$. As for the $L^{s_\eps}$--norm, by Cauchy-Schwarz inequality we have 
\[
|\Pi_\eps^\perp \phi|_{s_\eps} \leq |\phi|_{s_\eps}+ |\Pi_\eps(\phi)|_{s_\eps} \leq |\phi|_{s_\eps} + \frac{|\delta \partial_\delta w_\delta |_{s_\eps}}{\|\delta \partial_\delta w_\delta\|} \| \phi\|.
\]
By Lemma \ref{delta:lem}, there exists $c>0$ such that  $\|\delta \partial_\delta w_\delta\|\geq c$. On the other hand, by Corollary \ref{cor:estimates_with_bubbles}, $|\delta  \partial_\delta w_\delta(x)|\leq (U_{\delta,e_n}+U_{\delta,-e_n})$. Combining this with Lemma \ref{lemma:LpnormBubble}, we deduce that, as $\eps\to 0$,
\[
|\delta \partial_\delta w_\delta|_{s_\eps}\leq C(|U_{\delta,e_n}|_{s_\eps} + |U_{\delta,-e_n}|_{s_\eps})\leq C \eps^{-\eps \frac{(n-2)^2}{4n+2\eps(n-2)}}=O(1)
\]
and the proof is finished.
\end{proof}

We then decompose \eqref{w} in
\begin{equation}\label{w_projN}
 \Pi_\eps\left( w_\delta+\phi \right) = \Pi_\eps \circ K(f_\eps(w_\delta+\phi)),
\end{equation}
\begin{equation}\label{w_projNperp}
 \Pi_\eps^\perp\left( w_\delta+\phi \right) = \Pi_\eps^\perp \circ K(f_\eps(w_\delta+\phi)),
\end{equation}
and rewrite \eqref{w_projNperp} as
\begin{align}\label{eq:w_projN2}
    L_{d,\eps}\phi = N_{d,\eps}(\phi) + R_{d,\eps},
\end{align}
where
\begin{align*}
L_{d,\eps}\phi&:=\Pi_\eps^\perp\left(\phi-K(f'_\eps(w_\delta)\phi)\right),\\
N_{d,\eps}(\phi)&:=\Pi_\eps^\perp\circ K\left(f_\eps(w_\delta+\phi)-f_\eps(w_\delta)-f'_\eps(w_\delta)\phi\right),\\
 R_{d,\eps}&:=\Pi_\eps^\perp\left(K(f_\eps(w_\delta))-w_\delta\right).
\end{align*}

\section{Reduction to a finite dimensional problem}\label{sec:lin}
This section is dedicated to the treatment of \eqref{w_projNperp}, more precisely to the proof of the following result.
\begin{prop}\label{prop:C^1phi}
For every $0<\eta<1$ sufficiently small there exists $\eps_0>0$ and $C>0$ such that, whenever $\eps\in (0,\eps_0)$ and $d\in (\eta,1/\eta)$, there exists a unique function $\phi=\phi_{d,\eps}\in \Theta_{d,\eps}^\perp$ solving the equation
\begin{equation*}
    L_{d,\eps}\phi = N_{d,\eps}(\phi) + R_{d,\eps},
\end{equation*}
and satisfying
\begin{equation}\label{phi:bd}
\|\phi_{d,\eps}\|_{H_\eps({B})}=o(\eps^{1-\gamma}) \qquad \text{ for every small } \gamma>0,
\end{equation}
 as $\eps\to 0$, uniformly in $d\in (\eta,1/\eta)$. Moreover, the map  $
 (\eta,1/\eta) \to \Theta_{d,\eps}^\perp,$ ${d} \mapsto \phi_{{d},\eps}
$
is of class ${C}^1.$
\end{prop}
The  proof of this result has the following structure: first, in Subsection \ref{sec:3.1}, we check that $L_{d,\eps}$ is an invertible operator with continuous inverse in $\Theta_{d,\eps}^\perp$. Therefore, \eqref{eq:w_projN2} can be written as
\[
\phi=L^{-1}_{d,\eps}(N_{d,\eps}(\phi) + R_{d,\eps})
\]
and
\[
\|\phi\|_{H_\eps}\leq C(\|N_{d,\eps}(\phi)\|_{H_\eps} + \|R_{d,\eps}\|_{H_\eps}).
\]
In Subsection \ref{sec:3.2} we prove the estimate $\|R_{d,\eps}\|_{H_\eps}=O(\eps^{1-\gamma})$ for every $\gamma\in (0,1)$, and in Subsection \ref{sec:3.3} we conclude the proof of Proposition \ref{prop:C^1phi}, showing that the operator $N_{d,\eps}$ is of higher order with respect to $\phi$, which allows the use of a fixed point argument and the implicit function theorem. The main difficulties in these steps arise from the fact that we are dealing with a superlinear problem, which require delicate estimates in $L^{s_\eps}$-norms; moreover, the symmetry assumptions on both the domain and the functions play a crucial role in the proof of the invertibility of the linear operator $L_{d,\eps}$; see, for instance Remark \ref{domains} below.

\subsection{Estimates for the linear part \texorpdfstring{$\mathbf{L_{d,\eps}}$}{}}\label{sec:3.1}

\begin{prop}\label{lemma:linear_part}
For every $\eta\in (0,1)$ small enough there exists $\eps_0>0$ small, and $C>0$, such that if $\eps \in (0,\eps_0)$ then
\begin{equation}\label{17ott1}
\|{L}_{{d},\eps}({\phi})\|_{H_\eps} \geq C \|{\phi}\|_{H_\eps}  \qquad \forall {\phi} \in \Theta_{d,\eps}^\perp,\ {d}\in (\eta,1/\eta).
\end{equation}
Moreover, ${L}_{{d},\eps}$ is invertible in $\Theta_{d,\eps}^\perp$, with continuous inverse. 
\end{prop}
\begin{proof} We adapt the proof of \cite[Lemma 3.1]{MP10} to our setting. We argue by contradiction, assuming there exists $\eta\in (0,1)$, $d_k\in (\eta,1/\eta)$ with $d_k\to d_*$, $\eps_k\to 0$ and $\phi_k\in \Theta_{d_k,\eps_k}^\perp$ such that
\begin{align}\label{phik}
\|\phi_k\|_{H_{\eps_k}}=1,\quad \|h_k\|_{H_{\eps_k}}\to 0,
\end{align}
where 
\begin{align}\label{hk}
h_k:={L}_{{d_k},\eps_k}({\phi_k})\in \Theta_{d_k,\eps_k}^\perp    
\end{align}
 Denote also $\delta_k:=d_k\eps_k\to 0$ and $Z_{\delta}:=\delta \partial_\delta W_\delta$.  By \eqref{hk},
\begin{equation}\label{eq:linear_aux1}
\phi_k-K(f'_{\eps_k}(w_{\delta_k})\phi_k)= h_k + z_k
\end{equation}
with $z_k\in \Theta_{d_k,\eps_k}$, which means there exists $c_k\in \R$ such that 
\begin{align}\label{Ec3}
z_k=c_k PZ_{\delta_k}.    
\end{align}

\noindent Step 1. Check that $\|z_k\|_{H_{\eps_k}}\to 0$.

We test \eqref{eq:linear_aux1} with $z_k$. Since $\langle \phi_k,z_k\rangle=\langle h_k,z_k\rangle=0$,
\[
\|z_k\|^2=-\int_{B} f'_{\eps_k}(w_{\delta_k})\phi_k z_k.
\]
Using that $z_k=c_kPZ_{\delta_k}$,
\begin{align}\label{Ec1}
c_k^2 \|PZ_{\delta_k}\|^2 =  - c_k\int_{B} f'_{\eps_k}(w_{\delta_k})\phi_k P Z_{\delta_k}.
\end{align}
Now, by Lemma \ref{delta:lem},
\begin{align}\label{eq:aux_stat}
\|PZ_{\delta_k}\|^2= \kappa+o(1),
\end{align}
for some $\kappa>0$, while 
\[
0=\langle z_k,\phi_k\rangle=c_k \int_{B} \delta_k (f_0'(U_{\delta_k,e_n})\partial_\delta U_{\delta_k,e_n}-f_0'(U_{\delta_k,-e_n})\partial_\delta U_{\delta_k,-e_n})\phi_k,
\]
where $\partial_\delta U_{\delta_k,\pm e_n}:= \partial_\delta U_{\delta,\pm e_n}|_{\delta=\delta_k}$. Then,
\begin{align}
\int_{B} f'_{\eps_k}(w_{\delta_k})\phi_k P Z_{\delta_k} =&\int_{B} f'_{\eps_k}(w_{\delta_k})\phi_k (P Z_{\delta_k}-Z_{\delta_k})+\int_{B} (f'_{\eps_k}(w_{\delta_k})-f_0'(w_{\delta_k}))\phi_k Z_{\delta_k}\notag\\
				&+\int_{B} (f'_{0}(w_{\delta_k})Z_{\delta_k} -f'_0(U_{\delta_k,e_n})\delta_k \partial_\delta U_{\delta_k,e_n}+f'_0(U_{\delta_k,-e_n})\delta_k \partial_\delta U_{\delta_k,-e_n})\phi_k.\label{Ec2}
\end{align}

Note that, by Corollary \ref{cor:estimates_with_bubbles}, Lemmas \ref{lemma:LpnormBubble} and \ref{lemmaA},  \eqref{phik} and by \eqref{eq:aux_stat},
\begin{align*}
T_1&:=\left|\int_{B} f'_{\eps_k}(w_{\delta_k})\phi_k (P Z_{\delta_k}-Z_{\delta_k})\right|
\leq |f'_{\eps_k}(w_{\delta_k})|_{\frac{n}{2}}|\phi_k|_{2^*} |P Z_{\delta_k}-Z_{\delta_k}|_{2^*}=o(1).
\end{align*}
Similarly, by Lemmas \ref{lemma:LpnormBubble} and \ref{lemmaA}, \eqref{phik}, and \eqref{pa U con U},
\begin{align*}
T_2&:=\left|\int_{B} (f'_{\eps_k}(w_{\delta_k})-f_0'(w_{\delta_k}))\phi_k Z_{\delta_k}\right|
\leq |f'_{\eps_k}(w_{\delta_k})-f_0'(w_{\delta_k})|_{\frac{n}{2}}|\phi_k|_{2^*} |Z_{\delta_k}|_{2^*}=o(1).
\end{align*}
Finally, recalling that $w_{\delta_k}=PW_{\delta_k}=P(U_{\delta_k,e_n}-U_{\delta_k,-e_n})$ and that $Z_{\delta_k}=\delta_k \partial_\delta U_{\delta_k,e_n}-\delta_k \partial_\delta U_{\delta_k,-e_n}$,
\begin{align*}
    T_3&:=\left|\int_{B} (f'_{0}(w_{\delta_k})Z_{\delta_k} -f'_0(U_{\delta_k,e_n})\delta_k \partial_\delta U_{\delta_k,e_n}+f'_0(U_{\delta_k,-e_n})\delta_k \partial_\delta U_{\delta_k,-e_n})\phi_k\right|\\
    &\leq\int_{{B}} |f'_{0}(w_{\delta_k}) -f'_0(U_{\delta_k,e_n})||\delta_k \partial_\delta U_{\delta_k,e_n}||\phi_k|+\int_{{B}}|f'_{0}(w_{\delta_k})-f'_0(U_{\delta_k,-e_n})||\delta_k \partial_\delta U_{\delta_k,-e_n}||\phi_k|\\
    &\leq 2|U_{\delta_k,e_n}|_{2^*}|\phi_k|_{2^*}|f'_{0}(w_{\delta_k}) -f'_0(U_{\delta_k,e_n})|_{\frac{n}{2}}=o(1),
\end{align*}
where we used \eqref{pa U con U}, \eqref{phik} and Lemma \ref{lemmaA}.

Then, by \eqref{Ec1}, \eqref{Ec2}, and the fact that $T_1+T_2+T_3=o(1)$, we have that $c_k\to 0$. Therefore, by \eqref{Ec3} and \eqref{eq:aux_stat}, we conclude that $\|z_k\|_{H_{\eps_k}}\to 0$.

\medbreak

\noindent Step 2. Prove that
\begin{equation}\label{eq:Step2}
\liminf_{k\to\infty} \int_{B} f_{\eps_k}'(w_{\delta_k}) u_k^2>0,
\end{equation}
where $u_k:=\phi_k-h_k-z_k$, which satisfies
\begin{equation}\label{eq:u_k1}
u_k=K(f'_{\eps_k}(w_{\delta_k})\phi_k)=K(f'_{\eps_k}(w_{\delta_k})u_k + f'_{\eps_k}(w_{\delta_k})(h_k+z_k))
\end{equation}
or, equivalently,
\begin{equation}\label{eq:uk2}
\begin{cases}
-\Delta u_k=f'_{\eps_k}(w_{\delta_k})\phi_k =f'_{\eps_k}(w_{\delta_k})u_k + f'_{\eps_k}(w_{\delta_k})(h_k+z_k)\text{ in } {B},\\
\int_{B} u_k=0,\quad \partial_\nu u_k=0 \text{ on } \partial {B},
\end{cases}
\end{equation}
and $\|u_k\|_{H_{\eps_k}}\to 1$ (by Step 1). We claim that 
\begin{equation}\label{eq:lowerboundu_k}
    \liminf_{k\to\infty} \|u_k\|>0.
\end{equation} 
Indeed, by using \eqref{phik}, \eqref{eq:u_k1}, Lemmas \ref{lemma:Kcontinuous} and \ref{lemmaA}, and Step 1,
\begin{align*}
    |u_k|_{s_{\eps_k}}&\leq C \left( |f'_{\eps_k}(w_{\delta_k})u_k|_{\frac{ns_{\eps_k}}{n+2s_{\eps_k}}} + |f'_{\eps_k}(w_{\delta_k})(h_k+z_k)|_{\frac{ns_{\eps_k}}{n+2s_{\eps_k}}}  \right)\\
                    & \leq C |f'_{\eps_k}(w_{\delta_k})|_{\frac{(p+1)ns_{\eps_k}}{(p+1)(n+2s_{\eps_k})-ns_{\eps_k}}}|u_k|_{p+1}+ C |f'_{\eps_k}(w_{\delta_k})|_{\frac{n}{2}}|h_k+z_k|_{s_{\eps_k}}\leq C\|u_k\|+o(1),
\end{align*}
where we have used that $\frac{(p+1)ns_{\eps_k}}{(p+1)(n+2s_{\eps_k})-ns_{\eps_k}}=\frac{n}{2}+O(\eps_k)$ and $|f'_{\eps_k}(w_{\delta_k})|_{\frac{(p+1)ns_{\eps_k}}{(p+1)(n+2s_{\eps_k})-ns_{\eps_k}}}=O(1)$ as $k\to\infty$, by Corollary \ref{cor:estimates_with_bubbles} and Lemma \ref{lemma:LpnormBubble}.  Therefore, if $\|u_k\|\to 0$, then also $|u_k|_{s_{\eps_k}}\to 0$ and $\|u_k\|_{H_{\eps_k}}\to 0$, a contradiction. Hence, the claim \eqref{eq:lowerboundu_k} is true.

Next, testing \eqref{eq:uk2} with $u_k$, we obtain
\begin{equation}\label{eq:identityu_k}
\|u_k\|=\int_{B} f'_{\eps_k}(w_{\delta_k})u_k^2 + \int_{B} f'_{\eps_k}(w_{\delta_k})(h_k+z_k)u_k.    
\end{equation}
Since
\begin{align*}
\left|\int_{B} f'_{\eps_k}(w_{\delta_k})(h_k+z_k)u_k\right|\leq |f'_{\eps_k}(w_{\delta_k})|_{\frac{n}{2}}|h_k+z_k|_{2^*} |u_k|_{2^*} \leq C\|h_k+z_k\|_{H_{\eps_k}} |u_k|_{H_{\eps_k}}\to 0,
\end{align*}
combining this with \eqref{eq:lowerboundu_k} and \eqref{eq:identityu_k} yields directly \eqref{eq:Step2}, which is the goal of Step 2.

\medbreak

\noindent Step 3.  
Since ${B}$ is a smooth domain, there is an extension operator $E:H^1({B})\to H^1(\R^N)$ such that $Eu_k=u_k$ in ${B}$ and
\begin{align*}
    \|\nabla Eu_k\|_{L^2(\R^n)}\leq \|Eu_k\|_{H^1(\R^n)}\leq C\|u_k\|_{H^1({B})},
\end{align*}
for some constant $C=C({B})>0$.  Since $u_k$ has zero average in ${B}$, the Poincaré-Wirtinger inequality implies that 
\begin{align*}
    \|\nabla Eu_k\|_{L^2(\R^N)}\leq C'\|\nabla u_k\|_{L^2({B})},
\end{align*}
for some constant $C'=C'({B})>0$.

We identify $u_k$ with its extension. Define
\begin{align}\label{eq:blouwup_en}
\hat u_k(y)=\delta_k^\frac{n-2}{2}u_k(\delta_k y+e_n),\quad \text{ and } \quad {\Omega_k}:=\frac{{B}-e_n}{\delta_k}=B_{\frac{1}{\delta_k}}(-\frac{e_n}{\delta_k}),    
\end{align}
which is such that
\[
\int_{\R^n} |\nabla \hat u_k|^2 \leq C'\int_{B} |\nabla u_k|^2 = C' \int_{{\Omega_k}}|\nabla \hat u_k|^2.
\]
Therefore, $(\hat u_k)$ is a bounded sequence in $\mathcal{D}^{1,2}(\R^n)$ and, passing to a subsequence, there is $u_0\in D^{1,2}(\R^n)$ such that
\begin{align}\label{wc}
    \hat u_k\to  u_0\quad  \text{ weakly in } \mathcal{D}^{1,2}(\R^n), \text{ strongly in } L^q_\text{loc}(\R^n) \text{ for all } q\in [2,2^*).
\end{align}
We want to prove that $u_0=0$.

We have, for $x=\delta_k y+e_n\in {B}$ and $y\in {\Omega_k}$, that
\begin{align*}
\begin{cases}
-\Delta \hat u_k(y)=\delta_k^2 f'_{\eps_k}(w_{\delta_k}(x)) \hat u_k(y) + \delta_k^2f'_{\eps_k}(w_{\delta_k}(x))(\hat h_k(y)+\hat z_k(y)),\\
\int_{{\Omega_k}} \hat u_k=0,
\end{cases}
\end{align*}
where $\hat h_k(y)=\delta_k^\frac{n-2}{2}h_k(\delta_k y+e_n)$, $\hat z_k(y)=\delta_k^\frac{n-2}{2}h_k(\delta_k y+e_n)$.
Moreover, for $y\in \partial {\Omega_k}$,
\[
\partial_\nu \hat u_k(y)=\delta_k^\frac{n}{2}\nabla u_k (\delta_k y+e_n)\cdot \nu(y)=\delta_k^\frac{n}{2}\nabla u_k (\delta_k y+e_n)\cdot \frac{\delta_k y+e_n}{|\delta_k y+e_n|}= \delta_k^2 \nabla u_k (x)\cdot \frac{x}{|x|}=0.
\]

Let $\varphi \in C^\infty_c (\overline{\R^n_+})$. Then, by dominated convergence, Corollary \ref{cor:estimates_with_bubbles}, \eqref{wc}, and by $\| h_k\|,\| z_k\|\to 0$,
\begin{align*}
\int_{\R^n_+}\nabla u_0 \nabla \varphi =&\lim_{k\to\infty}\int_{{\Omega_k}}\nabla \hat u_k \nabla \varphi = 
\lim_{k\to\infty}\int_{{\Omega_k}}\delta_k^2f'_{\eps_k}(w_{\delta_k}(\delta_ky+e_n))\hat u_k(y) \varphi(y)\, dy\\
                &+\lim_{k\to \infty}\int_{{\Omega_k}} \delta_k^2 f'_{\eps_k}(w_{\delta_k}(\delta_ky+e_n))(\hat h_k(y)+\hat z_k(y))\varphi(y)\, dy = \int_{\R^n_+}f'_{0}(U_{1,0}) u_0 \varphi.
\end{align*}

Then $u_0$ would be a solution of
\begin{align*}
-\Delta u_0 = f_0'(U_{1,0})u_0\quad \text{ in }\R^n_+,\qquad \partial_\nu u_0=0\quad \text{ on }\partial \R^n_+.
\end{align*}
Identifying $u_0$ with its even reflection with respect to $\partial \R^n_+$, we obtain that $u_0$ is a solution of  
\begin{align}\label{eq:linearized_limit_problem}
-\Delta u_0 = f_0'(U_{1,0})u_0\quad \text{ in }\R^n,\qquad u_0\in D^{1,2}(\R^n).
\end{align}
Since $u_k\in H_{\eps_k}$, then $u_k$ is even with respect to $x_1,\ldots, x_{n-1}$. This is preserved under the change of variables $y\mapsto \delta_k y+e_n$ and under the even reflection with respect to $\partial \R_+^n$, so we have that 
\begin{align}\label{a:obs}
\text{$\hat u_k$ is also even in the coordinates $x_1,\ldots, x_{n-1}$.}
\end{align}
Therefore, $u_0$ is also even with respect to $x_1,\ldots, x_{n-1}$.  By \eqref{Vdef}, we have that $u_0=c V$ with $V$ as in \eqref{Vdef}.  We claim that $c=0$.  Indeed, since $u_k$ is odd with respect to $x_n$ and ${B}=B_1(0)$, we have that
\begin{align*}
 \int_{{B}} u_k f_0'(U_{\delta_k,e_n})\delta_k \partial_\delta U_{\delta_k,e_n}
 =-\int_{{B}} u_k f_0'(U_{\delta_k,-e_n})\delta_k \partial_\delta U_{\delta_k,-e_n},
\end{align*}
and therefore,
\begin{align*}
c\int_{\R^n}|\nabla V|^2
&=2\int_{\R^n_+}\nabla V\nabla u_0
=2\int_{{\Omega_k}}\nabla V\nabla \hat u_k+o(1)\\
&=2\int_{{\Omega_k}}\hat u_k f'(U_{1,0})V+o(1)
=2\int_{{B}} u_k f_0'(U_{\delta_k,e_n})\delta_k \partial_\delta U_{\delta_k,e_n}+o(1)\\
&=\int_{B} u_k [f'_0(U_{\delta_k,e_n})\delta_k \partial_\delta U_{\delta_k,e_n}-f'_0(U_{\delta_k,-e_n})\delta_k \partial_\delta U_{\delta_k,-e_n}]+o(1)\\
&=\langle u_k , PZ_{\delta_k}\rangle+o(1)
=\langle z_k , PZ_{\delta_k}\rangle+o(1)=o(1),
\end{align*}
where we used that $u_k=\phi_k-h_k-z_k$ in ${B}$ and that $\|z_k\|_{H_{\eps_k}}\to 0$, by Step 1.  Therefore, $c=0$ and $u_0=0$.

\medbreak

\noindent Step 4.  Finally, we check that
\begin{equation}\label{eq:Step4}
\liminf_{k\to\infty} \int_{B} f_{\eps_k}'(w_{\delta_k}) u_k^2=0,
\end{equation}
which contradicts \eqref{eq:Step2} and concludes the proof of the proposition. Observe that, using Corollary \ref{cor:estimates_with_bubbles},
\[
\left|\int_{B} f_{\eps_k}'(w_{\delta_k}) u_k^2 \right| \leq C\int_{B} |PW_{\delta_k}|^{\frac{4}{n-2}+\eps_k} |u_k|^2\leq C \int_{B} (U_{\delta_k,e_n}^{\frac{4+\eps_k(n-2)}{n-2}} +U_{\delta_k,-e_n}^{\frac{4+\eps_k(n-2)}{n-2}})|u_k|^2
\]

We split ${B}=\omega^+\cup \omega^-\cup ({B} \setminus (\omega^+\cup \omega^-))$. On ${B} \setminus (\omega^+\cup \omega^-)$, we have
\begin{align*}
\int_{{B}\setminus\{\omega^+\cup \omega_-\}} (U_{\delta_k,e_n}^{\frac{4+\eps_k(n-2)}{n-2}} +U_{\delta_k,-e_n}^{\frac{4+\eps(n-2)}{n-2}})u_k^2 \leq C\delta_k^{\frac{4+\eps_k(n-2)}{n-2}}\int_{B} u_k^2=O(\delta_k^{\frac{4+\eps_k(n-2)}{2}})=o(1).
\end{align*}
On $\omega^+$, considering the extension of $u_k$ to the whole $\R^n$ and considering the blowup sequence $\hat u_k$ defined in \eqref{eq:blouwup_en}, we have
\begin{align*}
    \int_{\omega^+} (U_{\delta_k,e_n}^{\frac{4+\eps_k(n-2)}{n-2}} +U_{\delta_k,-e_n}^{\frac{4+\eps(n-2)}{n-2}})u_k^2 &= \int_{\omega^+} U_{\delta_k,e_n}^{\frac{4+\eps_k(n-2)}{n-2}}u_k^2 + o(1)\leq C \int_{{\Omega_k}} U_{1,0}^\frac{4+\eps_k(n-2)}{n-2}\hat u_{k}^2+o(1)\\
            & \leq C\int_{\R^n} \frac{1}{(1+|y|^2)^2} \hat u_{k}^2 +o(1)=o(1),
\end{align*}
since $\hat u_{k}^2\rightharpoonup 0$ in $L^{\frac{n}{n-2}}(\R^n)$ and $\frac{1}{(1+|y|^2)^2}\in L^\frac{n}{2}(\R^n)$.

Considering now the blowup sequence at $- e_n$ given by $y\mapsto \delta_k^{\frac{n-2}{2}} u_k (\delta_k y- e_k)$, we prove in an analogous way that
\[
 \int_{\omega^-} (U_{\delta_k,e_n}^{\frac{4+\eps_k(n-2)}{n-2}} +U_{\delta_k,-e_n}^{\frac{4+\eps(n-2)}{n-2}})u_k^2 =o(1). \qedhere
\]
\end{proof}

\begin{rmk}\label{domains}
In the previous proof, a key ingredient is that \eqref{a:obs} holds, namely, that if $\varphi\in H_\eps$ then $\hat \varphi(y):= \delta^{\frac{n-2}{2}}\varphi(\delta_k y+e_n)$ is even in the coordinates $x_1,\ldots,x_{n-1}$.  This property imposes some restrictions when considering the more general case of \eqref{form} where $2k$ bubbles are considered.  In particular, it explains why the blow-up points $x_i$ need to be positioned at regular angles and why the signs must be alternating. 
\end{rmk}

\subsection{Estimates for the zero order term
\texorpdfstring{$\mathbf{R_{d,\eps}}$}{}}\label{sec:3.2}

In this subsection we prove the following asymptotic bound for $R_{d,\eps}$.
\begin{prop}\label{prop:estimate_remainderR}
Let $\gamma\in (0,1)$, $0<\eta<1$ and $\delta=d\eps$. Then
\begin{align*}
\|R_{d,\eps}\|_{H_\eps}=O(\eps^{1-\gamma})
\end{align*}
as $\eps\to 0$, uniformly in $d\in (\eta,1/\eta)$.
\end{prop}

Observe that
\begin{align*}
R_{d,\eps}&=\Pi_\eps^\perp(K(f_\eps(w_\delta))-w_\delta)
=\Pi_\eps^\perp(K(f_\eps(w_\delta)-f_0(w_\delta)+f_0(w_\delta))-w_\delta)\\
&=\Pi_\eps^\perp(K(f_\eps(w_\delta)-f_0(w_\delta)))+\Pi_\eps^\perp(K(f_0(w_\delta))-w_\delta).
\end{align*}
Since $\|\cdot\|_{H_\eps} = \|\cdot\|+|\cdot|_{s_\eps}$ and using Lemma \ref{lemma:projection_continuous},
\begin{align}
\|R_{d,\eps}\|_{H_\eps} \leq
				&C\left( \|K(f_\eps(w_\delta)-f_0(w_\delta))\| +  \|K(f_0(w_\delta))-w_\delta\| \nonumber \right. \\
				& \left.+ |K(f_\eps(w_\delta)-f_0(w_\delta))|_{s_\eps}+ |K(f_0(w_\delta))-w_\delta|_{s_\eps} \right).\label{eq:estimateR_d_eps}
\end{align}

Thus we need to estimate 
\begin{align*}
    \|K(f_0(w))-w_\delta\|,\qquad  \|K(f_\eps(w_\delta)-f_0(w_\delta))\|,
\end{align*} 
(see Lemmas \ref{lemma:I_1} and \ref{lemma:|i*(f_eps(w_delta)-f_0(w_delta))|} below) as well as 
\begin{align*}
|K(f_\eps(w_\delta)-f_0(w_\delta))|_{s_\eps},\qquad |K(f_0(w_\delta))-w_\delta|_{s_\eps}
\end{align*}
(Lemmas \ref{lemma:estimates_snorm1} and \ref{lemma:|K(f_0(w_delta))-w_delta |_s_eps} below). The proof of Proposition \ref{prop:estimate_remainderR} will follow directly from this.

\begin{lemma}\label{lemma:I_1} Let $\eta\in (0,1)$, $d>0$ and $\delta=d\eps$. We have that, for every $\gamma\in (0,1)$,
\[
\|K(f_0(w))-w_\delta\|\leq C |f_0(w_\delta) -f_0(U_{\delta,e_n})+f_0(U_{\delta,-e_n})|_{\frac{p+1}{p}}
=\begin{cases}
O(\delta|\log \delta|^\frac{1}{4}) & \text{ if } n=4,\\
O(\delta)=O(\eps) & \text{ if } n\geq 5,
\end{cases}=O(\eps^{1-\gamma})
\]
as $\eps\to 0$, uniformly in $d\in (1/\eta,\eta)$.
\end{lemma}
\begin{proof}
Let $v_\delta=K(f_0(w_\delta))$ and recall that $w_\delta=K(f_0(U_{\delta,e_n})-f_0(U_{\delta,-e_n}))$, that is, $v_\delta$ and $w_\delta$ solve
\begin{align*}
 -\Delta v_\delta &= f_0(w_\delta)=f_0(P(U_{\delta,e_n}-U_{\delta,-e_n}))\quad \text{ in }{B},\qquad 
 \partial_\nu v_\delta=0\quad \text{ on }\partial {B},\qquad 
 \int_{B} v_\delta = 0,\\
 -\Delta w_\delta &= f_0(U_{\delta,e_n})-f_0(U_{\delta,-e_n})\quad \text{ in }{B},\qquad 
 \partial_\nu w_\delta=0\quad \text{ on }\partial {B},\qquad 
 \int_{B} w_\delta = 0.
\end{align*}
Then, recalling that $p+1=\frac{2n}{n-2}$ and by \eqref{eq:Kinequality} in Lemma \ref{lemma:Kcontinuous},
\begin{align}\label{eq:estimate||u-v||}
\|v_\delta-w_\delta\|\leq C |f_0(w_\delta) -f_0(U_{\delta,e_n})+f_0(U_{\delta,-e_n})|_{\frac{p+1}{p}}.
\end{align}
We have
\begin{multline*}
|f_0(w_\delta) -f_0(U_{\delta,e_n})+f_0(U_{\delta,-e_n})|_{\frac{p+1}{p}}\\
	\leq |f_0(w_\delta) - f_0(W_\delta)|_{\frac{p+1}{p}}+ |f_0(W_\delta) - f_0(U_{\delta,e_n})+f_0(U_{\delta,-e_n})|_{\frac{p+1}{p}}.
\end{multline*}
The claim now follows from Lemma \ref{Lemma:I_1aux}.
\end{proof}

\begin{lemma}\label{lemma:|i*(f_eps(w_delta)-f_0(w_delta))|}
Let $\eta\in (0,1)$, $d>0$ and $\delta=d\eps$. Then, for every $\gamma\in(0,1)$,
\[
\|K(f_\eps(w_\delta)-f_0(w_\delta))\|
\leq C|f_\eps(w_\delta)-f_0(w_\delta)|_\frac{p+1}{p}=O(\eps^{1-\gamma}).
\]
as $\eps\to 0$, uniformly in $d\in (1/\eta,\eta)$.
\end{lemma}
\begin{proof}
By \eqref{eq:Kinequality} in Lemma \ref{lemma:Kcontinuous},
\[
\|K(f_\eps(w_\delta)-f_0(w_\delta))\|\leq C|f_\eps(w_\delta)-f_0(w_\delta)|_\frac{p+1}{p},
\]
and the statement now follows from Lemma \ref{lemma:|i*(f_eps(w_delta)-f_0(w_delta))|_aux}.
\end{proof}

\begin{lemma}\label{lemma:estimates_snorm1}
Let $\gamma,\eta\in (0,1)$, $d>0$ and $\delta=d\eps$. Then
\[
|K(f_\eps(w_\delta)-f_0(w_\delta))|_{s_\eps} = O(\eps^{1-\gamma})
\]
as $\eps\to 0$, uniformly in $d\in (1/\eta,\eta)$.
\end{lemma}
\begin{proof}
Let $\bar s:=s_{\bar \eps}=p+1+\bar \eps$ and $\eps<\bar \eps$.  By \eqref{eq:Kcontinuous} in Lemma \ref{lemma:Kcontinuous} (observe that $\bar s>n/(n-2)$):
\begin{align*}
|K(f_\eps(w_\delta)-f_0(w_\delta))|_{s_\eps} &\leq |K(f_\eps(w_\delta)-f_0(w_\delta))|_{\bar s} |{B}|^{\frac{\bar s-s_\eps}{\bar s \bar s}}\\
					&\leq C_{\bar \eps} |  f_\eps(w_\delta)-f_0(w_\delta) | _\frac{n\bar s}{n+2\bar s}.
\end{align*}
Taking now $\ell>0$ such that $\frac{n\bar s}{n+2\bar s}=\frac{(p+1)(1+\ell)}{p}$, we can now conclude from Lemma \ref{lemma:|i*(f_eps(w_delta)-f_0(w_delta))|_aux}.
\end{proof}

\begin{lemma}\label{lemma:|K(f_0(w_delta))-w_delta |_s_eps}
Let $\gamma,\eta \in (0,1)$, $d>0$ and $\delta=d\eps$. Then there exists $C>0$ such that
\[
|K(f_0(w_\delta))-w_\delta |_{s_\eps} =O( \eps^{1-\gamma})
\]
as $\eps\to 0$, uniformly in $d\in (1/\eta,\eta)$.
\end{lemma}
\begin{proof}
 Let $u:= K(f_0(w_\delta))$, $w=w_\delta$, and observe that $v:=u-w$ is a solution to
\[
-\Delta v=f_0(w_\delta)-f_0(U_{\delta,e_n})+f_0(U_{\delta,-e_n}) \text{ in } {B},\qquad \partial_\nu v=0 \text{ on } \partial {B},\qquad \int_{B} v=0,
\]
that is, $v=K(f_0(w_\delta)-f_0(U_{\delta,e_n})+f_0(U_{\delta,-e_n}))$.
Take $\bar \eps>0$ small and $\bar s=p+1+\bar \eps$. By \eqref{eq:Kcontinuous} in Lemma \ref{lemma:Kcontinuous}:
\begin{align*}
|u-w|_{s_\eps} &\leq |u-w|_{\bar s} |{B}|^{\frac{\bar s-s_\eps}{s_\eps \bar s}}\\
					& \leq C'_{\bar \eps} |f_0(w_\delta)-f_0(U_{\delta,e_n})+f_0(U_{\delta,-e_n})|_{\frac{n\bar s}{n+2\bar s}} \\
					&\leq C'_{\bar \eps} |f_0(w_\delta)-f_0(W_\delta)|_{\frac{n\bar s}{n+2\bar s}}  + | f_0(W_\delta) - f_0(U_{\delta,e_n})+f_0(U_{\delta,-e_n})|_{\frac{n\bar s}{n+2\bar s}}
\end{align*}
Writing $\frac{n\bar s}{n+2\bar s}=\frac{(p+1)(1+\gamma)}{p}$,
by Lemma \ref{Lemma:I_1aux} we deduce that, for every $\sigma>0$ small there exists $\gamma>0$ such that 
\[
|f_0(w_\delta) - f_0(W_\delta)|_{\frac{(p+1)(1+\gamma)}{p}}+ | f_0(W_\delta) - f_0(U_{\delta,e_n})+f_0(U_{\delta,-e_n})|_{\frac{(p+1)(1+\gamma)}{p}}=O(\eps^{1-\sigma})
\]
as $\eps\to 0$, and the proof of the lemma follows.
 \end{proof}

\begin{proof}[Proof of Proposition \ref{prop:estimate_remainderR}]

This follows from combining equation \eqref{eq:estimateR_d_eps} with Lemmas \ref{lemma:I_1}--\ref{lemma:|K(f_0(w_delta))-w_delta |_s_eps}
\end{proof}

\subsection{Estimates for the nonlinear part
\texorpdfstring{$\mathbf{N_{d,\eps}}$}{}. Conclusion of the proof of Proposition \ref{prop:C^1phi}}\label{sec:3.3}

In this subsection we conclude the proof of Proposition \ref{prop:C^1phi}. By Proposition \ref{lemma:linear_part}, we know that the linear operator $L_{d,\eps}$ is invertible. Therefore, equation \eqref{w_projNperp}, that is,
\begin{align*}
    L_{d,\eps}\phi = N_{d,\eps}(\phi) + R_{d,\eps},
\end{align*}
is equivalent to the fixed point problem
\begin{equation*}
\phi=L_{d,\eps}^{-1}(R_{d,\eps}+N_{d,\eps}(\phi))=:T_{d,\eps}(\phi).
\end{equation*}
By Proposition \ref{lemma:linear_part} we have
\[
\| T_{d,\eps}(\phi)\|_{H_\eps} \leq C (\| R_{d,\eps} \|_{H_\eps}+\| N_{d,\eps}(\phi)\|_{H_\eps}).
\]
In the previous subsection we have shown that, for every $\gamma\in (0,1)$, $\| R_{d,\eps} \|_{H_\eps}=O(\eps^{1-\gamma})$. Next we perform an estimate for the other term.

\begin{lemma}
Let $\eta \in (0,1)$, $d>0$ and $\delta=d\eps$. Then there exists $C>0$ such that
\[
\| N_{d,\eps}(\phi)\|_{H_\eps}\leq 
\begin{cases}
C (\|\phi\|^{p+\eps}_{H_\eps}+\|\phi\|_{H_\eps}^{p+\frac{2\eps}{n}}) & \text{ if } n>6,\\
C (\|\phi\|^{p+\eps}_{H_\eps}+\|\phi\|_{H_\eps}^{p+\frac{2\eps}{n}}+\|\phi\|_{H_\eps}^2) & \text{ if } n\in [4,6)
\end{cases}
\] 
as $\eps\to 0$, uniformly in $d\in (1/\eta,\eta)$.
\end{lemma}
\begin{proof}
Combining Lemmas \ref{lemma:Kcontinuous} and \ref{lemma:projection_continuous}
 yields the existence of $C>0$ such that 
\begin{multline*}
\| N_{d,\eps}(\phi)\|_{H_\eps} \leq C (|f_\eps(w_\delta+\phi)-f_\eps(w_\delta)-f'_\eps(w_\delta)\phi|_\frac{p+1}{p}+|f_\eps(w_\delta+\phi)-f_\eps(w_\delta)-f'_\eps(w_\delta)\phi|_\frac{ns_\eps}{n+2s_\eps}).
\end{multline*}
By Lemma \ref{lemma:TaylorInequalities},  there exists $C>0$ such that, for $\eps$ sufficiently small,
\[
|f_\eps(w_\delta+\phi)-f_\eps(w_\delta)-f'_\eps(w_\delta)\phi| \leq \begin{cases} C(|w_\delta|^{p-2+\eps}\phi^2+|\phi|^{p+\eps}) & \text{ if } n\in [4,6],\\
C|\phi|^{p+\eps} & \text{ if } n>6.
\end{cases}
\]
Therefore, for $n>6$, we see that
\begin{align*}
\| T_{d,\eps}(\phi)\|_{H_\eps} &\leq C (||\phi|^{p+\eps}|_\frac{p+1}{p}+||\phi|^{p+\eps}|_\frac{ns_\eps}{n+2s_\eps})\leq C ( |\phi|^{p+\eps}_\frac{(p+1)(p+\eps)}{p}+|\phi|_{s_\eps}^\frac{n+2s_\eps}{n})\\
&\leq C'(|\phi|^{p+\eps}_{s_\eps}+|\phi|_{s_\eps}^{p+\frac{2\eps}{n}})\leq C'(\|\phi\|^{p+\eps}_{H_\eps}+\|\phi\|_{H_\eps}^{p+\frac{2\eps}{n}}).
\end{align*}
For $n\in [4,6]$, one should also take into account the terms
\[
||w_\delta|^{p-2+\eps}\phi^2|_\frac{p+1}{p}\leq |w_\delta|^{p-2+\eps}_{\frac{(p-2+\eps)2ns_\eps}{(n+2)s_\eps-4n}} |\phi|_{s_\eps}^2
\]
and
\[
||w_\delta|^{p-2+\eps}\phi^2|_{\frac{ns_\eps}{n+2s_\eps}} \leq |w_\delta|_{\frac{(p-2+\eps)ns_\eps}{2s_\eps-n}}^{p-2+\eps} |\phi|_{s_\eps}^2.
\]
Since $\frac{(p-2+\eps)2ns_\eps}{(n+2)s_\eps-4n}=\frac{n}{2}+O(\eps)$ and $\frac{(p-2+\eps)ns_\eps}{2s_\eps-n}=\frac{n}{2}+O(\eps)$, we have that
$|w_\delta|^{p-2+\eps}_{\frac{(p-2+\eps)2ns_\eps}{(n+2)s_\eps-4n}}, |w_\delta|_{\frac{(p-2+\eps)ns_\eps}{2s_\eps-n}}^{p-2+\eps} =O(1)$, by Corollary \ref{cor:estimates_with_bubbles} and Lemma \ref{lemma:LpnormBubble}. This ends the proof.
\end{proof}

\smallbreak

\begin{proof}[Conclusion of the proof of Proposition \ref{prop:C^1phi}]
Let $\kappa$ be such that $ |R_{d,\eps}|\leq 2\kappa \sqrt{\eps}$ (take $\gamma=1/2$ in Proposition \ref{prop:estimate_remainderR}). Therefore, by taking
\[
\mathcal B:=\left\{\phi \in \Theta_{d,\eps}^\perp\ :\ \|\phi\|_{H_\eps}\leq \kappa \sqrt{\eps}\right\},
\]
we have $T_{d,\eps}(\mathcal B)\subset \mathcal{B}$ for sufficiently small $\eps>0$.
 Moreover, reasoning as in, for instance, \cite[pp. 18-19]{MP10}, we obtain the existence of $L\in (0,1)$ such that
\[
\|T_{d,\eps}(\phi_1-\phi_2)\|_{H_\eps} \leq L \| \phi_1- \phi_2 \|_{H_\eps}\ \text{ for every } \phi_1,\phi_2\in  \mathcal{B}.
\]
Therefore, by the Banach Fixed Point Theorem, given $\eta\in (0,1)$, $d\in (\eta,1/\eta)$, $\delta=d\eps$, for $\eps$ sufficiently small there exists a unique $\phi_{d,\eps}$, fixed point of $T_{d,\eps}$, that is, a unique solution of \ref{eq:w_projN2}. Now a  standard argument using the implicit function theorem (see for instance \cite[Lemma 3.3]{PistoiaTavares} for detailed computations in a related framework) yields that the map $d\mapsto \phi_{d,\eps}$ is of class $C^1$.
\end{proof}

\section{Expansion of the reduced functional} \label{sec:red}

Consider the functional $F_\eps:H_{\eps}\to \R$ defined by
\begin{align*}
  F_\eps(u)=\int_{B} \frac{|\nabla u|^2}{2} - \frac{|u|^{p+1+\eps}}{p+1+\eps}\, dx.
\end{align*}
and observe that critical points correspond to solutions to \eqref{super}. For each small $\eta$ fixed, let $\eps_0>0$ be as in Proposition \ref{prop:C^1phi}. Then for  $\eps\in (0,\eps_0)$ (and recalling that $\delta=d \eps$), we consider the reduced functional $J_\eps:(\eta,1/\eta)\to \R$ given by
\begin{align*}
J_\eps(d) := F_\eps(P W_{\delta}+\phi_{d, \eps})=F_\eps(P W_{d\eps}+\phi_{d, \eps}),
\end{align*}
where $\phi_{d, \eps}\in  \Theta_{d,\eps}^\perp$ is as in  Proposition \ref{prop:C^1phi}. The following results says that, whenever we find a critical point of $J_\eps$, we obtain a solution of \eqref{super} having the form $PW_{d\eps}+\phi_{d,\eps}$.

\begin{lemma}\label{eq:criticalpointJ}
For small $\eta>0$, $\eps\in (0,\eps_0)$ and $d\in (0,\eps)$, we have
\[
J_\eps'(d)=0 \iff F_\eps'(PW_{d\eps}+\phi_{d,\eps})=0.
\]
\end{lemma}
\begin{proof}
This a consequence of standard arguments, see for instance \cite[Proposition 2.2]{MP10} or \cite[Proposition 3.4]{PistoiaTavares}, \cite[Lemma 4.1]{PistoiaSoaveTavares}.
\end{proof}

For $n\geq 4$, recall that $\alpha_n=(n(n-2))^\frac{n-2}{4}$ and let 
\begin{equation}\label{fraks}
\begin{aligned}
\mathfrak{A}&= \int_{\R^n} U_{1,0}^{p+1} = \int_{\R^n} \frac{\alpha_n^{p+1}}{(1+|x|^2)^n}\, dx ,\\
\mathfrak{B}&= \int_{\partial\R^n_+} |U_{1,0}(y)|^{p+1}|y|^2=\alpha_n^{p+1} \int_{\R^{n-1}} \frac{|y|^2}{(1+|y|^2)^n}\, dy ,\\
\mathfrak{C}&=\alpha_n\int_{\R^{n-1}}
 \frac{|y|^2}{(1+|y|^2)^{n-1}},\\
\mathfrak{D}&=\int_{\R^n} U_{1,0}^{p+1} \log U_{1,0}.
\end{aligned}
\end{equation}

The main goal is to check that $J_\eps$ has an absolute minimizer in $(\eta,1/\eta)$ for sufficiently small $\eta$, therefore a critical point. In order to prove this, we compute the expansion of $J_\eps(d)$ as $\eps\to 0$.

\begin{thm}\label{epxansion:thm}
Given $0<\eta<1$ small we have
\begin{align*}
J_\eps(d)=\frac{\mathfrak{A}}{n}+ \frac{(n-2)^2}{4n}\mathfrak{A}\eps \log \eps +\Psi(d) \eps + o(\eps),
\end{align*}
as $\eps\to 0$ uniformly in $d\in (\eta,1/\eta)$, where
\begin{align}\label{eq:Psi(d)}
    \Psi(d):=\frac{n-2}{2n}\left(\frac{n-2}{2n}\mathfrak{A}-\mathfrak{D}\right)+\frac{(n-2)^2}{4n}\mathfrak{A} \log d - \left(\frac{n-2}{2}\mathfrak{C}+\frac{\mathfrak{B}}{n}\right) d.
\end{align}
\end{thm}

This section is devoted to the proof of this result, which we split in several lemmas.
\begin{lemma}\label{lemma:mainterm}
Given $0<\eta<1$ small, 
\begin{align*}
    J_\eps(d)=F_\eps(PW_\delta) + o(\eps)
\end{align*}
as $\eps\to 0$ uniformly in $d\in (\eta,1/\eta)$.
\end{lemma}
\begin{proof}
We argue as in \cite[Lemma 6.1]{MP10} to show that
\begin{align*}
F_\eps(PW_\delta+\phi_{d,\eps})-F_\eps(PW_\delta)=o(\eps)\qquad \text{ as }\eps\to 0.
\end{align*} 
Note that, by \eqref{phi:bd},
\begin{align*}
F_\eps(PW_\delta+\phi_{d,\eps})-F_\eps(PW_\delta)
=&\frac{1}{2}\|\phi_{d,\eps}\|^2+\int_{{B}}(U_{\delta,e_n}^p-U_{\delta,e_n}^p)\phi_{d,\eps}
\\
&-\frac{1}{p+1+\eps}\int_{{B}}|PW_\delta+\phi_{d,\eps}|^{p+1+\eps}-|PW_\delta|^{p+1+\eps}\\
&=o(\eps)+\int_{{B}}(U_{\delta,e_n}^p-U_{\delta,e_n}^p-|PW_\delta|^{p-1+\eps})PW_\delta\phi_{d,\eps}\\
&-\int_{{B}}\frac{|PW_\delta+\phi|^{p+1+\eps}}{p+1+\eps}-\frac{|PW_\delta|^{p+1+\eps}}{p+1+\eps}-|PW_\delta|^{p-1+\eps}PW_\delta\phi_{d,\eps}.
\end{align*}
Moreover, by \eqref{phi:bd} and Lemmas \ref{Lemma:I_1aux} and \ref{lemma:|i*(f_eps(w_delta)-f_0(w_delta))|_aux} with $\gamma=0$,
\begin{align*}
\left|\int_{{B}}(U_{\delta,e_n}^p-U_{\delta,e_n}^p-|PW_\delta|^{p+\eps})\phi_{d,\eps}\right|
\leq |U_{\delta,e_n}^p-U_{\delta,e_n}^p-|W_\delta|^{p+\eps}|_{\frac{2n}{n+2}}|\phi_{d,\eps}|_{2^*}=o(\eps).
\end{align*}
On the other hand, by the mean value theorem, there is $t=t(x,\eps)\in[0,1]$ such that
\begin{align*}
&\left|\int_{{B}}\frac{|PW_\delta+\phi|^{p+1+\eps}}{p+1+\eps}-\frac{|PW_\delta|^{p+1+\eps}}{p+1+\eps}-|PW_\delta|^{p-1+\eps}PW_\delta\phi_{d,\eps}\right|
\leq C\int_{B} |PW_\delta+t\phi_{d,\eps}|^{p-1+\eps}\phi^2\\
&\qquad\leq C'(\left||PW_\delta|^{p-1+\eps}|_{\frac{n}{2}}\right|\phi_{d,\eps}|^2_{2^*}+|\phi_{d,\eps}|_{s_\eps}^{s_\eps})=o(\eps),
\end{align*}
where we used again \eqref{phi:bd} and that $\left||PW_\delta|^{p-1+\eps}\right|_{\frac{n}{2}}=O(1)$ (by Corollary \ref{cor:estimates_with_bubbles} and Lemma \ref{lemma:LpnormBubble}).
\end{proof}

In the rest of the section we expand the leading term
\begin{align*}
 F_\eps(PW_\delta)
 =\int_{B} \frac{|\nabla PW_\delta|^2}{2} - \frac{|PW_\delta|^{p+1+\eps}}{p+1+\eps}\, dx,
\end{align*}
 We compute separately the expansions for
\begin{align*}
 \int_{B} \frac{|\nabla PW_\delta|^2}{2}\qquad\text{ and }\qquad 
 \int_{B}\frac{|PW_\delta|^{p+1+\eps}}{p+1+\eps}
 \end{align*}
in Subsections \ref{subsec:gradient} and \ref{subsec:nonlinear_term} respectively. We perform them in the $\delta$ variable, recalling that $\delta=d\eps$ for some $d>0$.

Recall also that $W_{\delta}:=U_{\delta,e_n}-U_{\delta,-e_n}$ and, by Lemma \ref{exp:lem},
 \begin{align*}
  PW_{\delta}(x)&=W_{\delta}(x)-\delta^{-\frac{n-4}{2}}\Big(\varphi_0\Big(\frac{e_n-x}{\delta}\Big)-\varphi_0\Big(\frac{e_n+x}{\delta}\Big)\Big)+\zeta_\delta(x),\qquad \zeta_\delta(x)=O(\delta^{\frac{6-n}{2}}(\log \delta)^\tau),
\end{align*}
where $\tau=0$ if $n\geq 5$, and $\tau=1$ if $n=4$.

\subsection{Expansion of the gradient term} \label{subsec:gradient}

Note that
\begin{align}
\int_{B} |\nabla P W_{\delta}|^2
&=  \int_{B} P W_{\delta}(-\Delta )P W_{\delta}
=  \int_{B} P W_{\delta}(U_{\delta,e_n}^p-U_{\delta,-e_n}^p)\notag\\
&=  \int_{B} \left(W_{\delta}(x)-\delta^{-\frac{n-4}{2}}\Big(\varphi_0\Big(\frac{e_n-x}{\delta}\Big)-\varphi_0\Big(\frac{e_n+x}{\delta}\Big)\Big)+\zeta_\delta(x)\right)(U_{\delta,e_n}^p-U_{\delta,-e_n}^p)\notag\\
&=  I_1-I_2+I_3,\label{grad:term}
\end{align}
where
\begin{align*}
I_1&:=\int_{B} W_{\delta}(x)(U_{\delta,e_n}^p-U_{\delta,-e_n}^p),\\
I_2&:= \int_{B} \delta^{-\frac{n-4}{2}}\Big(\varphi_0\Big(\frac{e_n-x}{\delta}\Big)-\varphi_0\Big(\frac{e_n+x}{\delta}\Big)\Big)(U_{\delta,e_n}^p-U_{\delta,-e_n}^p),\\
I_3&:=\int_{B} \zeta_\delta(x)(U_{\delta,e_n}^p-U_{\delta,-e_n}^p).
\end{align*}

\begin{prop}\label{prop:exp:1} We have the expansion
\begin{align*}
\int_{B} |\nabla P W_{\delta}|^2
&=I_1-I_2+I_3
=\mathfrak{A}+(-\mathfrak{B}+(n-2)\mathfrak{C})\delta+o(\delta),\qquad \text{ as }\delta\to 0,
\end{align*}
where $\mathfrak{A}$, $\mathfrak{B}$, and $\mathfrak{C}$ are given in \eqref{fraks}.
\end{prop}

The proof of this proposition follows directly from the following three lemmas.

\begin{lemma}We have
\begin{align*}
I_1=\int_{B} W_{\delta}(x)(U_{\delta,e_n}^p-U_{\delta,-e_n}^p)
  = 
  \mathfrak{A}-\mathfrak{B}\delta+o(\delta),\qquad \text{ as }\delta\to 0,
   \end{align*}
   where $\mathfrak{A}$ and $\mathfrak{B}$ are given in \eqref{fraks}.
\end{lemma}
\begin{proof}
Note that
\begin{align*}
  I_1=\int_{B} W_{\delta}(x)(U_{\delta,e_n}^p-U_{\delta,-e_n}^p)
  = \int_{B} U_{\delta,e_n}^{p+1}-U_{\delta,e_n}U_{\delta,-e_n}^p
  -U_{\delta,-e_n}U_{\delta,e_n}^p
  +U_{\delta,-e_n}^{p+1}.
\end{align*}
The result follows from Lemmas \ref{AB:lem} and \ref{inter:lem}.
\end{proof}

\begin{lemma} We have
\begin{align*}
I_2&=-(n-2)\mathfrak{C}\,\delta+o(\delta),
\quad \text{ as }\delta\to 0,
\end{align*}
where $\mathfrak{C}$ is given in \eqref{fraks}.
\end{lemma}
\begin{proof}
Note that
\begin{align*}
    I_2
    = \delta^{-\frac{n-4}{2}}\int_{B}
    \varphi_0\Big(\frac{e_n-x}{\delta}\Big)U_{\delta,e_n}^p
    -\varphi_0\Big(\frac{e_n+x}{\delta}\Big)U_{\delta,e_n}^p
    -\varphi_0\Big(\frac{e_n-x}{\delta}\Big)U_{\delta,-e_n}^p
    +\varphi_0\Big(\frac{e_n+x}{\delta}\Big)U_{\delta,-e_n}^p.
\end{align*}
Therefore, the result follows from Lemmas \ref{I2:1} and \ref{I2:2} and the fact that $n\geq 4$.
\end{proof}

\begin{lemma}
Let $\tau=0$ if $n\geq 5$, and $\tau=1$ if $n=4$. We have
\begin{align*}
    I_3:=\int_{B} \zeta_\delta(x)(U_{\delta,e_n}^p-U_{\delta,-e_n}^p)=O(\delta^2(\log \delta)^\tau)=o(\delta) \qquad \text{as $\delta\to 0$}.
\end{align*}

\end{lemma}
\begin{proof}
Let $D:=\omega_+-e_n$  and observe that, since $\zeta_\delta=O(\delta^\frac{6-n}{2}(\log \delta)^\tau)$ uniformly in ${B}$ as $\delta\to 0$ (by Lemma \ref{exp:lem}),
\begin{align}
\int_{\omega_+} |\zeta_\delta(x)|U_{\delta,e_n}^p
&\leq C\alpha_n^p\delta^\frac{6-n}{2}(\log \delta)^\tau\int_{\omega_+}
\frac{\delta^{\frac{n+2}{2}}}{(\delta^2+|x-e_n|^2)^{\frac{n+2}{2}}}\notag\\
&= C\alpha_n^p\delta^\frac{6-n}{2}(\log \delta)^\tau\int_{\delta^{-1}D}
\frac{\delta^{\frac{n-2}{2}}}{(1+|x|^2)^{\frac{n+2}{2}}}
=C\alpha_n^p\delta^{2}(\log \delta)^\tau\int_{\delta^{-1}D}
\frac{1}{(1+|x|^2)^{\frac{n+2}{2}}}\notag\\
&\leq C\alpha_n^p\delta^{2}(\log \delta)^\tau\omega_n\left(1+
\int_{1}^\frac{1}{2\delta}r^{-3}\ dr
\right)=O(\delta^2(\log \delta)^\tau).\label{3}
\end{align}
On the other hand,
\begin{align}
\int_{{B}\backslash \omega_+} |\zeta_\delta(x)|U_{\delta,e_n}^p
&\leq C\alpha_n^p\delta^\frac{6-n}{2} (\log \delta)^\tau\int_{{B}\backslash \omega_+}
(4\delta)^{\frac{n+2}{2}}=O(\delta^4(\log \delta)^\tau).\label{4}
\end{align}
The claim follows from \eqref{3} and \eqref{4} and from the fact that 
\[
\int_{B} \zeta_\delta(x)U_{\delta,e_n}^p
=\int_{B} \zeta_\delta(x)U_{\delta,-e_n}^p.  \qedhere
\]
\end{proof}

\subsection{The nonlinear term} \label{subsec:nonlinear_term}

Next, we focus on an expansion for the nonlinear term 
\begin{align*}
 \frac{1}{p+1+\eps}\int_{B}|PW_\delta|^{p+1+\eps},
\end{align*}
where we recall we are using  the Ansatz $\delta=d\eps$ for some $d>0$. 
\begin{prop}\label{prop:exp:2}
Given $\eta\in (0,1)$, we have the expansion
\begin{align*}
&\frac{1}{p+1+\varepsilon}\int_{{B}}|PW_\delta|^{p+1+\eps}\\
&=\frac{n-2}{2n}(\mathfrak{A}-
\mathfrak{B}\,d\eps+2n \mathfrak{C}\, d\eps
)
+\left(
-\frac{(n-2)^2}{4n}\mathfrak{A}\, \eps \log|d\eps|
+\frac{n-2}{2n}\mathfrak{D}\,\eps
-\frac{(n-2)^2}{4n^2}\eps\mathfrak{A}
\right)
+o(\eps),
\end{align*}
as $\eps\to 0^+$ uniformly in $d\in (\eta,1/\eta)$.
\end{prop}

We argue as in \cite[Proposition A.1]{RW05}. We start by noting that
\begin{align}\label{pp1exp}
\frac{1}{p+1+\varepsilon}
=\frac{1}{p+1}-\frac{1}{(p+1)^2}\varepsilon+o(\varepsilon)
=\frac{n-2}{2n}-\frac{(n-2)^2}{4n^2}\eps+o(\eps).
\end{align}
Moreover,
\begin{align}\label{J1J2exp}
\int_{{B}}|PW_\delta|^{p+1+\eps}
=\int_{{B}}|PW_\delta|^{p+1}+|PW_\delta|^{p+1}(|PW_\delta|^\varepsilon-1)
=:J_1+J_2.
\end{align}
We obtain expansions for each of these terms. For convenience, within this subsection we denote
\begin{align*}
\psi(x):=&W_{\delta}(x)- PW_{\delta}(x)\\
=&\delta^{-\frac{n-4}{2}}\Big(\varphi_0\Big(\frac{e_n-x}{\delta}\Big)-\varphi_0\Big(\frac{e_n+x}{\delta}\Big)\Big)-\zeta_\delta(x),\qquad\zeta_\delta(x)=O(\delta^{\frac{6-n}{2}}). 
\end{align*}

\begin{lemma}\label{J1lem} We have
\begin{align}
&\int_{\omega_+}|W_{\delta}(x)-\psi(x)|^{p+1}=\int_{\omega_-}|W_{\delta}(x)-\psi(x)|^{p+1}
=\frac{\mathfrak{A}}{2}+(-\frac{\mathfrak{B}}{2}+n \mathfrak{C})\delta+o(\delta),\label{5}
\end{align}
where $\mathfrak{A}$, $\mathfrak{B}$, and $\mathfrak{C}$ are defined in \eqref{fraks}. In particular,
\begin{align*}
J_1=\int_{{B}}|PW_\delta|^{p+1}
=\int_{{B}}|W_{\delta}(x)-\psi(x)|^{p+1}=\mathfrak{A}+(-\mathfrak{B}+2n \mathfrak{C})\delta+o(\delta).
\end{align*}
\end{lemma}
\begin{proof}
Let $D:=\omega_+-e_n$, then
\begin{align*}
&\int_{\omega_+}|W_{\delta}(x)-\psi(x)|^{p+1}
=\int_{\omega_+}|U_{\delta,e_n}(x)-U_{\delta,-e_n}(x)-\psi(x)|^{\frac{2n}{n-2}}\\
&=\int_{\delta^{-1}D}\left|
\alpha_n \frac{\delta^{-\frac{n-2}{2}}}{(1+|x|^2)^{\frac{n-2}{2}}}
-\psi(\delta x+e_n)\right|^{\frac{2n}{n-2}}\delta^{n}+o(\delta)\\
&=\int_{\delta^{-1}D}\left|
\alpha_n \frac{1}{(1+|x|^2)^{\frac{n-2}{2}}}
-
\Big(\varphi_0(-x)-\varphi_0\left(\frac{2e_n+\delta x}{\delta}\right)\Big)\delta
-\zeta_\delta(\delta x+e_n)\delta^{\frac{n-2}{2}}
\right|^{\frac{2n}{n-2}}+o(\delta)\\
&=\int_{\delta^{-1}D}\left|
U_{1,0}
-
\varphi_0(-x)\delta
\right|^{\frac{2n}{n-2}}+o(\delta)
=\int_{\delta^{-1}D}\left|U_{1,0}\right|^{\frac{2n}{n-2}}
-\delta \frac{2n}{n-2}\left|U_{1,0}\right|^{\frac{n+2}{n-2}}\varphi_0(-x)
+o(\delta),
\end{align*}
where we use Lemma \ref{exp:lem}, the decay estimates for $\varphi_0$ and $\zeta_\delta$, respectively inequalities \eqref{A6} and \eqref{eq:estimate_PW-W}. The integral $\int_{\delta^{-1}D}\left|U_{1,0}\right|^{\frac{2n}{n-2}}$ is expanded in the proof of Lemma \ref{AB:lem} and we have that
\begin{align*}
\int_{\delta^{-1}D}\left|U_{1,0}\right|^{\frac{2n}{n-2}}=\frac{\mathfrak{A}}{2}-\frac{\mathfrak{B}}{2}\delta + o(\delta).
\end{align*}
Moreover, since $U_{1,0}$ is even, $\partial_{x_n} U_{1,0}=0$ on $\partial \R^n_+$ and
\begin{align*}
&\delta \frac{2n}{n-2}\int_{\delta^{-1}D}\left|U_{1,0}\right|^{\frac{n+2}{n-2}}\varphi_0(-x)
=\delta \frac{2n}{n-2}\int_{\R^n_+}(-\Delta)U_{1,0}\varphi_0(x)+o(\delta) =\delta \frac{2n}{n-2}\int_{\partial \R^n_+} U_{1,0} \partial_\nu \varphi_0\\
&=-\delta \frac{2n}{n-2}\int_{\R^{n-1}}
\alpha_n\frac{n-2}{2}\frac{|y|^2}{(1+|y|^2)^{n-1}}
+o(\delta)
=-\delta n \alpha_n\int_{\R^{n-1}}
\frac{|y|^2}{(1+|y|^2)^{n-1}}
+o(\delta).
\end{align*}
By symmetry, this establishes \eqref{5}. Finally, note that, by Corollary \ref{cor:estimates_with_bubbles},               
\[
\int_{{B}\backslash (\omega_+\cup\omega_-)}|W_{\delta}(x)-\psi(x)|^{p+1}
= \int_{{B}\backslash (\omega_+\cup\omega_-)}|PW_{\delta}(x)|^{p+1}\leq C\int_{{B}\backslash (\omega_+\cup\omega_-)}(|U_{\delta,e_n}|+|U_{\delta,-e_n}|)^{\frac{2n}{n-2}}=o(\delta). \qedhere
\]
\end{proof}

\begin{lemma}\label{J2lem}
We have, as $\eps\to 0$,
\begin{align*}
J_2&=-\eps \frac{n-2}{2}\log(\delta){\mathfrak A}
+\eps{\mathfrak D}+o(\eps).
\end{align*}
\end{lemma}
\begin{proof}
Let $D:=\omega_+-e_n$ and recall that $\delta=d\eps$. We start by claiming that
\begin{align}\label{eq:J_2}
J_2
=\int_{{B}}|PW_\delta|^{p+1}(|PW_\delta|^\varepsilon-1)
=\eps\int_{{B}}|PW_\delta|^{p+1}\log|PW_\delta| + o(\delta).
\end{align}
Indeed, by reasoning like in the proof of Lemma \ref{lemma:f_eps_expansion},
\[
|t|^{p+1+\eps}=|t|^{p+1}+\eps |t|^{p+1}\log|t|+\eps^2r_\eps(t), \quad \text{
where }\quad 
|r_\eps(t)|\leq \frac{1}{2}(|t|^{p+1}+|t|^{p+1+\eps})(\log |t|)^2.
\]
Therefore,
\[
J_2
=
\eps\int_{{B}}|PW_\delta|^{p+1}\log|PW_\delta| +\eps^2\int_{B} r_\eps(PW_\delta);
\]
now the claim \eqref{eq:J_2} follows by reasoning like in the proof of Lemma \ref{lemma:|i*(f_eps(w_delta)-f_0(w_delta))|_aux}, recalling that $|PW_\delta|\leq |U_{\delta,e_n}|+|U_{\delta,-e_n}|$ and using Lemma \ref{lemma:LpnormBubble}.

Then,
\begin{align*}
\eps \int_{\omega_+}|PW_\delta|^{p+1}\log|PW_\delta|&=\eps \int_{\omega_+}|W_\delta-\psi|^{p+1}\log|W_\delta-\psi|\\
&=\eps \int_{\omega_+}|W_\delta|^{p+1}\log|W_\delta|+o(\eps)\\
&=\eps \int_{\omega_+}|U_{\delta,e_n}|^{p+1}\log|U_{\delta,e_n}|+o(\eps)\\
&=\eps \int_{\delta^{-1}D}|\delta^{-\frac{n-2}{2}}U_{1,0}|^{\frac{2n}{n-2}}\log|\delta^{-\frac{n-2}{2}}U_{1,0}|\delta^n\\
&=\eps \int_{\delta^{-1}D}|U_{1,0}|^{\frac{2n}{n-2}}\log|\delta^{-\frac{n-2}{2}}U_{1,0}|+o(\eps)\\
&=\eps \int_{\delta^{-1}D}|U_{1,0}|^{\frac{2n}{n-2}}
\left(-
\frac{n-2}{2}\log|\delta|
+\log|U_{1,0}|
\right)+o(\eps)\\
&=-\eps \frac{n-2}{2}\log|\delta|\int_{\delta^{-1}D}|U_{1,0}|^{\frac{2n}{n-2}}
+\eps\int_{\delta^{-1}D}|U_{1,0}|^{\frac{2n}{n-2}}\log|U_{1,0}|+o(\eps),
\end{align*}
where we have used the estimate \eqref{eq:last_estimate} with $q=p+1$, together with Lemma \ref{exp:lem}, Corollary \ref{cor:estimates_with_bubbles} and Lemma \ref{lemma:LpnormBubble}.

One can argue similarly for the integral over $\omega_-$ and the integral over ${B}\backslash (\omega_+\cup\omega_-)$ is $o(\eps)$ as $\eps\to 0$. Therefore, since (by Lemma \ref{exp:lem})
\begin{align*}
\int_{\delta^{-1}D}\left|U_{1,0}\right|^{\frac{2n}{n-2}}=\frac{\mathfrak{A}}{2}-\frac{\mathfrak{B}}{2}\delta + o(\delta)
\end{align*}
and 
\begin{align*}
    \int_{\delta^{-1}D}|U_{1,0}|^{\frac{2n}{n-2}}\log|U_{1,0}|= \frac{1}{2}\int_{\R^n}|U_{1,0}|^{\frac{2n}{n-2}}\log|U_{1,0}|+o(1)=\frac{\mathfrak{D}}{2}+o(1),
\end{align*}
we have that
 \begin{align*}
 J_2
&=-\eps\frac{n-2}{2}\log|d\eps|\mathfrak A
+\eps{\mathfrak D}+o(\eps),
 \end{align*}
 as claimed.
\end{proof}

\begin{proof}[Proof of Proposition \ref{prop:exp:2}]
From \eqref{J1J2exp} and Lemmas \ref{J1lem} and \ref{J2lem}, we have
\begin{align*}
\int_{{B}}|PW_\delta|^{p+1+\eps}
=\mathfrak{A}-\mathfrak{B}\,d\eps +2n \mathfrak{C}\, d\eps-
\frac{n-2}{2}\mathfrak{A}\, \eps\log|d\eps|
+\mathfrak{D}\eps
+ o(\eps).
\end{align*}
Combining this with \eqref{pp1exp}, we deduce
\begin{align*}
&\frac{1}{p+1+\varepsilon}\int_{{B}}|PW_\delta|^{p+1+\eps}\\
&=\left(\frac{n-2}{2n}-\frac{(n-2)^2}{4n^2}\eps+o(\eps)\right)(\mathfrak{A}-\mathfrak{B}\,d\eps +2n \mathfrak{C}\, d\eps
-
\frac{n-2}{2}\mathfrak{A}\, \eps \log|d\eps|
+\mathfrak{D}\,\eps
+ o(\eps))\\
&=\frac{n-2}{2n}(\mathfrak{A}-
\mathfrak{B}\,d\eps +2n \mathfrak{C}\, d\eps
)
+\left(
-\frac{(n-2)^2}{4n}\mathfrak{A}\, \eps \log|d\eps|
+\frac{n-2}{2n}\mathfrak{D}\,\eps
-\frac{(n-2)^2}{4n^2}\eps\mathfrak{A}
\right)
+o(\eps),
\end{align*}
as $\eps\to 0^+$.
\end{proof}

\begin{proof}[Proof of Theorem \ref{epxansion:thm}]
The expansion follows directly from Lemma \ref{lemma:mainterm} and Propositions \ref{prop:exp:1} and \ref{prop:exp:2}.
\end{proof}

\section{Proof of the main theorems} \label{sec:final_proof}

\begin{lemma}
Let $\Psi$ be as in \eqref{eq:Psi(d)}. This function has a unique critical point at
\begin{equation}\label{eq:d_n}
d_*=\frac{(n-2)^2\mathfrak{A}}{2d(n(n-2)\mathfrak{C}+2\mathfrak{B})},
\end{equation}
 which is a global maximum.
\end{lemma}
\begin{proof}
Note that $\Psi(d)\to -\infty$ as  $d\to 0^+$ and as $d\to +\infty$. Therefore, $\Psi$ achieves a global maximum. Since 
\[
\Psi'(d)=\frac{(n-2)^2 \mathfrak{A} }{4n} \frac{1}{d}-\left(\frac{n-2}{2}\mathfrak{C}+\frac{\mathfrak{B}}{n}\right),
\]
the point $d_*$ given in \eqref{eq:d_n} corresponds to the \emph{unique} critical point of $\Psi$, which is a global maximum.
\end{proof}
\begin{proof}[Proof of Theorem \ref{thm:main:ball}]

Let $d_*>0$ be the unique global maximum of $\Psi$ and let $\eta\in(0,1)$ be  such that $d_*\in (\eta,1/\eta).$ Then $\Psi(\eta),\Psi(1/\eta)< \Psi(d^*)$. On the other hand, for such $\eta$ we know from  Theorem \ref{epxansion:thm} that
\begin{align*}
J_\eps(d)=\frac{\mathfrak{A}}{n}+ \frac{(n-2)^2}{4n}\mathfrak{A}\eps \log \eps +\Psi(d) \eps + o(\eps),
\end{align*}
as $\eps\to 0$, uniformly in $(\eta,1/\eta)$. Then there exists $\eps_0>0$ such that
\[
J_\eps(\eta)< J_\eps(d^*)
\quad \text{ and } \quad J_\eps(1/\eta)< J_\eps(d^*)
\]
for $\eps\in (0,\eps_0)$. Therefore,  $J_\eps$ has an interior maximum in $(\eta,1/\eta)$, hence a critical point. We can now conclude by using Lemma \ref{eq:criticalpointJ}.
\end{proof}

\section{Open problems}\label{op:sec}
We believe that the following are some interesting open questions. 
\begin{enumerate}
\item[$(i)$] If ${B}$ is a ball, is is true that \eqref{super:intro} has a solution for all exponents $q>p$ or is there a \emph{Neumann critical exponent} after which there are no solutions?
\item[$(ii)$] The solutions given by Theorem \ref{thm:main:ball} are most likely \emph{not} of least-energy type. Indeed, we believe that the Morse index of our solution could be at least 2 (arguing as in \cite[Theorem 1]{BahriLiRey}).  Is there a least-energy solution in the supercritical regime? 
\end{enumerate}
Finally, we observe that our approach does not guarantee that the solutions given by Theorem \ref{thm:main:ball} are bounded (and therefore classical) solutions. However, if one uses  weighted H\"older 
norms as in  \cite{PFM03} (the pioneering paper  concerning the slightly-supercritical problem), instead of the classical Sobolev norms,  one can most likely build, via the same Lyapunov-Schmidt reduction, regular solutions.

\appendix

\section{An expansion for \texorpdfstring{$PW_\delta$}{} }\label{app:expansion}

Recall that ${B}=B_1(0)$, $W_\delta:=U_{\delta,e_n}-U_{\delta,-e_n}$ and $P W_{\delta}$ is the solution of
\begin{align*}
-\Delta P W_{\delta}=-\Delta  U_{\delta,e_n}+\Delta  U_{\delta,-e_n}\ \hbox{in}\ {B},\qquad \partial_\nu P W_{\delta}=0\ \hbox{on}\ \partial{B}, \qquad \int\limits_{{B}} P W_{\delta}dx=0.
\end{align*}
Let $\R^n_+:=\{x\in \R^n\::\: x_n>0\}$ and let $\varphi_0$ be the solution of \eqref{phieq}, that is,
\begin{align*}
-\Delta \varphi_0 = 0 \text{ in }\R^n_+,\qquad
  \frac{\partial \varphi_{0}}{\partial x_n}
  =\alpha_n\frac{n-2}{2}\frac{|x'|^2}{(1+|x'|^2)^\frac{n}{2}} \text{ on }\partial \R^n_+,\qquad
 \varphi_0\to 0  \text{ as }|x|\to \infty.
\end{align*}
Using the Poisson kernel for the halfspace and the fact that $n\geq 4$, we obtain the representation
\begin{align}\label{phinull}
 \varphi_0(x)=-\frac{\alpha_n}{\omega_{n}}\int_{\R^{n-1}}\frac{|y|^2}{(1+|y|^2)^\frac{n}{2}|x-y|^{n-2}}\ dy,
\end{align}
where $\omega_{n}$ denotes the measure of the unit sphere in $\R^n$, $|\cdot|$ is the Euclidean norm in $\R^n$, $x=(x_1,\ldots, x_n)\in \R^n_+$, and $y=(y_1,\ldots, y_{n-1},0)$, see for instance \cite[Theorem 4]{CPGreenhalfspace}.  We note, however, that the formula  (27) in \cite{CPGreenhalfspace} has a sign mistake (indeed, if $\varphi_0$ were positive, it would violate the maximum principle, since $\varphi_0$ would be a positive harmonic function with an interior maximum). Moreover, observe that formula \eqref{phinull} is consistent with the case $n=3$  \cite[p. 120]{Kev2000}, or with the case of a Neumann problem in a bounded $C^1$ domain, see \cite[p.165]{FJR78}.

Rewriting \eqref{phinull} as
\[
\varphi_0(x)=-\frac{\alpha_n}{\omega_{n} |x|^{n-3}}\int_{\R^{n-1}} \frac{|z|^2}{\left(\frac{1}{|x|^2}+|z|^2\right)^\frac{n}{2} \left|\frac{x}{|x|}-z\right|^{n-2}}  \, dz
\] we deduce its asymptotic behaviour, which yields the existence of $C>0$ such that
\begin{align}\label{A6}
 |\varphi_0(x)|\leq \frac{C}{(1+|x|)^{n-3}},
 \qquad
 |\nabla \varphi_0(x)|\leq \frac{C}{(1+|x|)^{n-2}},
 \qquad
 |D^2 \varphi_0(x)|\leq \frac{C}{(1+|x|)^{n-1}}
\end{align}
for $x\in \R^n_+$. 

Before we proceed to the proof of Lemma \ref{exp:lem}, we present the following auxiliary result.

\begin{lemma}\label{lemma:asymptotics_aux}
Let 
\[
z_\delta:= \partial_\nu \left[ U_{\delta,0}(x)-\delta^{\frac{4-n}{2}} \varphi_0\left(\frac{x}{\delta}\right) \right],\qquad x\in \partial B_1(e_n).
\]
Then, as $\delta\to 0$,
\begin{align}\label{eq:asymptotics_aux1}
z_\delta=  O\left(\frac{\delta^\frac{4-n}{2}}{(1+|\frac{x'}{\delta}|)^{n-3}}\right) \text{ on } \partial B_1(e_n)\cap B_1(0),\qquad z_\delta=  O\left(\delta^\frac{n-2}{2}\right) \text{ on } \partial B_1(e_n) \cap  B_1^c(0).
\end{align}
and
\begin{align}\label{eq:asymptotics_aux2} 
\partial_\delta z_\delta=  O\left(\frac{\delta^\frac{2-n}{2}}{(1+|\frac{x'}{\delta}|)^{n-3}}\right) \text{ on } \partial B_1(e_n)\cap B_1(0),\qquad \partial_\delta z_\delta=  O\left(\delta^\frac{n-4}{2}\right) \text{ on } \partial B_1(e_n) \cap  B_1^c(0).
\end{align}

\end{lemma}
\begin{proof}
We follow the ideas from \cite[Lemma A.1]{RW05}. 
\smallbreak

\noindent Step 1. We start by proving estimates on $\partial B_1(e_n)\cap B_1^c(0)$. In this set, we have
\begin{align*}
|\partial_\nu U_{\delta,0}| &\leq |\nabla U_{\delta,0}|= \alpha_n (n-2)\delta^{\frac{n-2}{2}}\frac{|x|}{(\delta^2+|x|^2)^\frac{n}{2}} =O(\delta^\frac{n-2}{2}),\\
|\partial_\delta \partial_\nu U_{\delta,0}| &\leq |\nabla \partial_\delta U_{\delta,0}|\leq C\delta^\frac{n-4}{2}\frac{|x|}{(\delta^2+|x|^2)^\frac{n}{2}}+ C\delta^{\frac{n-2}{2}}\frac{|x|}{(\delta^2+|x|^2)^\frac{n+2}{2}}=O(\delta^\frac{n-4}{2}).
\end{align*}
and, by using \eqref{A6},
\begin{align*}
\left|\delta^{\frac{4-n}{2}}\partial_\nu \left[ \varphi_0\left(\frac{x}{\delta}\right) \right]\right| &\leq \delta^{\frac{2-n}{2}}  \left|\nabla \varphi_0\left(\frac{x}{\delta}\right)\right| \leq \frac{C\delta^{\frac{2-n}{2}}  }{\left(1+\left|\frac{x}{\delta}\right|\right)^{n-2}}=O(\delta^\frac{n-2}{2}).\\
\left|\partial_\delta \left(\delta^{\frac{4-n}{2}}\partial_\nu \left[ \varphi_0\left(\frac{x}{\delta}\right) \right]\right)\right| &\leq \delta^{-\frac{n}{2}}\left|\nabla \varphi_0\left(\frac{x}{\delta}\right)\right| + \delta^{-\frac{n+2}{2}}\left|D^2\varphi_0\left(\frac{x}{\delta}\right)\right|=O(\delta^\frac{n-4}{2}).
\end{align*}
Therefore, $z_\delta=  O\left(\delta^\frac{n-2}{2}\right)$ and $\partial_\delta z_\delta=  O\left(\delta^\frac{n-4}{2}\right) $ on  $\partial B_1(e_n) \cap  B_1^c(0)$.
 
\smallbreak

\noindent Step 2. Next we focus on the (complementary) set $\partial B_1(e_n)\cap B_1(0)$, which can be described through the identity 
\begin{align}\label{rho}
x_n= \rho(x'): = 1-\sqrt{1-|x'|^2}=\frac{|x'|^2}{2}+O(|x'|^3)\qquad \text{ as $|x'|\to 0$}.
\end{align}
In particular, this shows that the principal curvatures for the unitary ball at any point of its boundary are $k_i=1$ for $i=1,\ldots,n-1$. Moreover, the exterior unitary normal on this part of the boundary is given by the expression
\begin{align}\label{nu}
 \nu(x)=\frac{( \nabla \rho,-1 )}{\sqrt{1+|\nabla \rho|^2}}=\frac{(\frac{x'}{\sqrt{1-|x'|^2}},-1 )}{\sqrt{1+\frac{|x'|^2}{1-|x'|^2}}}=(x',-\sqrt{1-|x'|^2}).
\end{align}
Combining  \eqref{rho} with \eqref{nu}, for $x\in \partial B_1(e_n)\cap B_1(0)$ we have
\begin{align}
 \partial_\nu U_{\delta,0}(x)
 &=\partial_\nu \alpha_n \frac{\delta^{\frac{n-2}{2}}}{(\delta^2+|x|^2)^{\frac{n-2}{2}}}=-\alpha_n (n-2)\delta^{\frac{n-2}{2}}\frac{\langle x,\nu \rangle}{(\delta^2+|x|^2)^\frac{n}{2}} \nonumber \\
 &=-\alpha_n(n-2)\delta^{\frac{n-2}{2}}\frac{|x'|^2-x_n\sqrt{1-|x'|^2}}{(\delta^2+|x|^2)^\frac{n}{2}}  \nonumber \\
  &=-\alpha_n(n-2)\delta^{\frac{n-2}{2}}\frac{|x'|^2-\sqrt{1-|x'|^2}+1-|x'|^2}{(\delta^2+|x|^2)^\frac{n}{2}} \nonumber  \\
  &=-\alpha_n(n-2)\delta^{\frac{n-2}{2}}\frac{1-\sqrt{1-|x'|^2}}{(\delta^2+|x|^2)^\frac{n}{2}} \nonumber  \\
  &=- \frac{\alpha_n(n-2)}{2}\frac{\delta^{\frac{n-2}{2}}|x'|^2}{(\delta^2+|x|^2)^\frac{n}{2}} + O\left(	\frac{\delta^{\frac{n-2}{2}}|x'|^3}{(\delta^2+|x|^2)^\frac{n}{2}}\right)  \nonumber  \\
&  - \frac{\alpha_n(n-2)}{2}\frac{\delta^{\frac{n-2}{2}}|x'|^2}{(\delta^2+|x'|^2)^\frac{n}{2}} + O\left(	\frac{\delta^{\frac{n-2}{2}}|x'|^3}{(\delta^2+|x'|^2)^\frac{n}{2}}\right),\label{eq:s1}
\end{align}
where in the last equality we have used the fact that 
\[
\left|\frac{1}{(\delta^2+|x|^2)^\frac{n}{2}}-\frac{1}{(\delta^2+|x'|^2)^\frac{n}{2}}  \right|\leq \frac{n}{2}\frac{x_n^2}{(\delta^2+|x'|^2)^\frac{n+2}{2}}=O\left(\frac{|x'|^4}{(\delta^2+|x'|^2)^{n+2}}  \right)
\]
Let $\eta_\delta:\overline{B_1(e_n)}\to\R$ be given by
\begin{align*}
\eta_\delta(x)=\delta^{-\frac{n-4}{2}}\varphi_0\left(\frac{x}{\delta}\right)\qquad \text{ for }\delta>0.
\end{align*}
By \eqref{phieq}, for $x\in \partial B_1(e_n)\cap B_1(0)$,
 \begin{align}
  \partial_\nu \eta_\delta(x)
  &=\nabla \eta_\delta(x)\cdot \nu(x)\\
  &=\partial_{-e_n}\eta_\delta(x',0)+((\nabla \eta_\delta(x)-\nabla \eta_\delta(x',0))\cdot \nu(x)+\nabla \eta_\delta(x',0)\cdot (\nu(x)+e_n))\\
  &=- \frac{\alpha_n(n-2)}{2}\frac{\delta^{\frac{n-2}{2}}|x'|^2}{(\delta^2+|x'|^2)^\frac{n}{2}}+ O\left(\frac{\delta^\frac{n-2}{2}}{(\delta+|x'|)^{n-3}}\right), \label{eq:s2}
  \end{align}
as $\delta\to 0$, since, by \eqref{A6} and \eqref{nu},
 \begin{align*}
\delta^{\frac{n-2}{2}}| \nabla \eta_\delta(x',0)\cdot (\nu(x)+e_n))| &\leq |\nabla \varphi_0(\frac{x'}{\delta},0)| \, |\nu(x)+e_n| \\
			& \leq \frac{C|(x',1-\sqrt{1-|x'|^2})|}{(1+|\frac{x'}{\delta}|)^{n-2}}=O\left(\frac{\delta^{n-2}|x'|}{(\delta+|x'|)^{n-2}}\right)
 \end{align*}
 while, by \eqref{A6} and \eqref{rho}, there exist $t_1,\ldots, t_n\in (0,1)$ such that
  \begin{align}
\delta^{\frac{n-2}{2}} |\nabla \eta_\delta(x)-\nabla \eta_\delta(x',0)|&=|\nabla \varphi_0(\frac{x}{\delta})-\nabla\varphi_0(\frac{x'}{\delta},0)| \\
		&=\left|\left( \frac{\partial^2 \varphi_0}{\partial x_1\partial x_n} \left(\frac{x'}{\delta},t_1\frac{x_1}{\delta}\right),\ldots, \frac{\partial^2 \varphi_0}{\partial x_n^2} \left(\frac{x'}{\delta},t_n\frac{x_n}{\delta}\right) \right)  \right| \frac{|x_n|}{\delta}\\
		&\leq \frac{C|x_n|}{\delta(1+|\frac{x'}{\delta}|)^{n-1}}=O\left(\frac{\delta^{n-2}|x'|^2}{(\delta+|x'|)^{n-1}}\right). \label{eq:s3}
 \end{align}  
From and we deduce that
\[
z_\delta=  O\left(\frac{\delta^\frac{4-n}{2}}{(1+|\frac{x'}{\delta}|)^{n-3}}\right) \text{ on } \partial B_1(e_n)\cap B_1(0).
\]
As for the first asymptotic estimate in \eqref{eq:asymptotics_aux2}, it is enough to follow the previous steps, taking the derivative in \eqref{eq:s1} and \eqref{eq:s2} with respect to $\delta$ in and estimating accordingly. One should use the fact that
\[
\frac{\delta^\frac{n-2}{2}}{(\delta+|x'|)^{n-2}}=O\left(\frac{\delta^\frac{n-4}{2}}{(\delta+|x'|)^{n-3}}\right)
\] 
 and $|D^3\varphi_0(x)|\leq C/(1+|x|)^n$ (the last one being used in the estimates corresponding to \eqref{eq:s3}).
\end{proof}

For the pure Neumann problem in a ball, there is an explicit Green's function \cite{WirthGreen}. In particular, if $u$ is a solution of 
\begin{align*}
    -\Delta  u = 0\quad \text{ in }B,\qquad  \partial_\nu u=\varphi\qquad \text{ on }\partial B,
\end{align*}
then 
\begin{align*}
    u(x)=\int_{\partial B}G(x,y)\varphi(y)\, dy,
\end{align*}
where 
\begin{align}\label{G}
    G(x,y)=c_1|x-y|^{2-n}+c_2\int_0^{|y|}\frac{x\cdot \frac{y}{|y|}-\frac{1}{s}}{|sx-\frac{y}{|y|}|^n}+\frac{1}{s}\, ds-c_3|x|^2
\end{align}
for some explicit positive constants $c_1$, $c_2,$ and $c_3.$  Here, the second term can be decomposed in simpler explicit expressions, see \cite{WirthGreen}.  The function $G$ is bounded by a multiple of the fundamental solution; this follows from direct computations using the explicit formula \eqref{G} or from \cite[Proposition 9 in Appendix A]{DRWbook} (which also covers general smooth domains), see also \cite[Lemma 3.1]{WR05}.  In particular, there is $C>0$ such that
\begin{align}\label{Ceq}
|G(x,y)|\leq C|x-y|^{2-n}\qquad \text{ for all }x\in B,\ y\in\partial B,\ x\neq y.
\end{align}

For the proof of Lemma \ref{exp:lem}, we use the following estimate.

\begin{lemma}\label{WY:lem}
 Let $\sigma\in(0,n-2)$, then there is $C>0$ such that
 \begin{align*}
     \int_{\R^{n-1}}\frac{1}{|x-y|^{n-2}(1+|y|)^{1+\sigma}}\, dy \leq \frac{C}{(1+|x|)^{\sigma}}\qquad \text{ for all } x\in \R^n_+.
 \end{align*}
Moreover, when $\sigma=0$ and for a ball $B_{1/\delta}(0)\subset \R^{n-1}$,
 \[
 \int_{B_{1/\delta}(0)} \frac{1}{|x-y|^{n-2}(1+|y|)}\, dy=O(\log \delta) \qquad \text{ as } \delta\to 0.
 \]
\end{lemma}
\begin{proof}
The proof follows closely the argument in \cite[Lemma B.2]{WY10}.  We include a proof for the reader's convenience. Let $x\in \R^n_+$, $\sigma\in(0,n-2)$, $d=\frac{1}{2}|x|>0$, and let $B_d(0)\subset \R^{n-1}$. Using polar coordinates,
\begin{align*}
\int_{B_d(0)}\frac{1}{|x-y|^{n-2}(1+|y|)^{1+\sigma}}\, dy
\leq \frac{C}{d^{n-2}}\int_{B_d(0)}\frac{1}{(1+|y|)^{1+\sigma}}\, dy
\leq \frac{C}{d^{n-2}} d^{n-2-\sigma}=\frac{C}{d^\sigma}
\end{align*}
and 
\begin{align*}
\int_{B_d(x)}\frac{1}{|x-y|^{n-2}(1+|y|)^{1+\sigma}}\, dy
\leq \frac{C}{d^{1+\sigma}}\int_{B_d(0)}|y|^{2-n}\, dy
= \frac{C}{d^{1+\sigma}}d=\frac{C}{d^\sigma}.
\end{align*}

Next, let $y\in \R^{n-1}\backslash (B_d(0)\cup B_d(x))$, then
\begin{align*}
|x-y|\geq \frac{1}{2}|x|,\qquad |y|\geq \frac{1}{2}|x|.
\end{align*}

If $|y|\geq 2|x|$, then $|x-y|\geq |y|-|x|\geq \frac{1}{2}|y|$; therefore,
\begin{align*}
    \frac{1}{|x-y|^{n-2}(1+|y|)^{1+\sigma}}
    \leq \frac{C}{|y|^{n-2}(1+|y|)^{1+\sigma}}.
\end{align*}

If $|y|\leq 2|x|$, then 
\begin{align*}
    \frac{1}{|x-y|^{n-2}(1+|y|)^{1+\sigma}}
    \leq \frac{C}{|x|^{n-2}(1+|y|)^{1+\sigma}}
    \leq \frac{C}{|y|^{n-2}(1+|y|)^{1+\sigma}}.
\end{align*}

As a consequence,
\begin{align*}
    \frac{1}{|x-y|^{n-2}(1+|y|)^{1+\sigma}|}\leq \frac{C}{|y|^{n-2}(1+|y|)^{1+\sigma}}\qquad \text{ for all }y\in \R^{n-1}\backslash (B_d(0)\cup B_d(x)).
\end{align*}

Thus,
\begin{align*}
    \int_{\R^{n-1}\backslash (B_d(0)\cup B_d(x))}\frac{1}{|x-y|^{n-2}(1+|y|)^{1+\sigma}}\, dy
    &\leq \int_{\R^{n-1}\backslash (B_d(0)\cup B_d(x))}\frac{C}{|y|^{n-2}(1+|y|)^{1+\sigma}}\, dy\\
    &\leq C\int_d^\infty\frac{r^{n-2}}{r^{n-2}(1+r)^{1+\sigma}}\, dr=\frac{C}{d^\sigma}.
\end{align*}

If $\sigma=0$, the only difference in the proof is the next step: we use $B_{1/\delta}(0)\backslash (B_d(0)\cup B_d(x))$ instead of $\R^{n-1}\backslash (B_d(0)\cup B_d(x))$, obtaining that
\begin{align*}
    \int_{B_{1/\delta}(0)\backslash (B_d(0)\cup B_d(x))}\frac{1}{|x-y|^{n-2}(1+|y|)}\, dy
    &\leq C\int_d^{1/\delta}\frac{1}{r}\, dr=O(\log \delta).
\end{align*}
\end{proof}

\begin{proof}[Proof of Lemma \ref{exp:lem}]
We argue as \cite[Lemma A.1]{RW05}, with some modifications due to the fact that we project the difference of two bubbles and because our differential operator does not have a linear term\footnote{The operator $-\Delta+\mu$ with $\mu>0$ is considered in \cite[Lemma A.1]{RW05}, which with Neumann boundary conditions has important differences with respect to the case $\mu=0$.}.
Let us write

 \begin{align*}
 \zeta_\delta(x)=PW_{\delta}(x)-W_{\delta}(x)+\delta^{-\frac{n-4}{2}}\Big(\varphi_0\Big(\frac{e_n-x}{\delta}\Big)-\varphi_0\Big(\frac{e_n+x}{\delta}\Big)\Big)\quad \text{ for } x\in {B}=B_1(0),
 \end{align*}
which solves
\[
-\Delta \zeta_\delta=0 \text{ in } {B},\qquad  \partial_\nu \zeta_\delta= z_{1,\delta} + z_{2,\delta} \text{ on } \partial {B}, \qquad \int_{B} \zeta_\delta=0,
\]
where
\begin{align*}
z_{1,\delta}(x)= \partial_\nu\left[\delta^\frac{4-n}{2} \varphi_0\left(\frac{e_n-x}{\delta}\right)- U_{\delta,e_n}(x)\right],\qquad
z_{2,\delta}(x)= \partial_\nu \left[ U_{\delta,-e_n}(x) -\delta^\frac{4-n}{2}\varphi_0\left(\frac{e_n+x}{\delta}\right)\right].
\end{align*}

Let $\tau_1(x)=e_n-x$, which satisfies $\tau_1(B_1(0))=B_1(e_n)$, $\tau_1(0)=e_n$, $\tau_1(e_n)=0$. Moreover, for $x\in \partial B_1(0)$, we denote by $\nu(x)$ the exterior unitary normal on $\partial B_1(0)$ at $x$, and by $\nu(\tau_1(x))$ the exterior unitary normal on $\partial B_1(e_n)$ at $\tau_1(x)$. Note that $\nu(x)=-\nu(\tau_1(x))$. Let 
\[
\psi_\delta(y):= \delta^{\frac{4-n}{2}}\varphi_0\left(\frac{y}{\delta}\right)-U_{\delta,0}(y),\qquad \text{ for } y\in B_1(e_n).
\] Then
 \begin{align*}
z_{1,\delta}(x)= \partial_\nu \left[\psi_\delta(\tau_1(x))\right]=\nabla \psi_\delta \left[(\tau_1(x))\right]\cdot \nu (x)= \nabla \psi_\delta (\tau_1(x))\cdot \nu (\tau_1(x)).
 \end{align*}
 Hence, by \eqref{eq:asymptotics_aux1} in Lemma \ref{lemma:asymptotics_aux},
 \begin{align*}
z_{1,\delta}=  O\left(\frac{\delta^\frac{4-n}{2}}{(1+|\frac{x'}{\delta}|)^{n-3}}\right) \text{ on } \partial B_1(0)\cap B_1(e_n),\qquad z_{1,\delta}=  O\left(\delta^\frac{n-2}{2}\right) \text{ on } \partial B_1(0) \cap  B_1^c(e_n).
\end{align*}
By using instead the isometry $\tau_2(x)=x+e_n=x-(-e_n)$, we see in an analogous way that
\begin{align*}
z_{2,\delta}=  O\left(\frac{\delta^\frac{4-n}{2}}{(1+|\frac{x'}{\delta}|)^{n-3}}\right) \text{ on } \partial B_1(0)\cap B_1(-e_n),\qquad z_{2,\delta}=  O\left(\delta^\frac{n-2}{2}\right) \text{ on } \partial B_1(0) \cap  B_1^c(-e_n).
\end{align*}

Combining the previous estimates, we conclude that 
\begin{equation}\label{eq:estimate_x'}
\partial_\nu \zeta_\delta(x)=O\left(\frac{\delta^\frac{4-n}{2}}{(1+|\frac{x'}{\delta}|)^{n-3}}\right)\text{ in } \partial B_1(0)\cap B_1(\pm e_n),\qquad  \partial_\nu \zeta_\delta(x)=O(\delta^\frac{n-2}{2}) \text{ in } B_1(0)\cap B_1^c(\pm e_n)
\end{equation}
and, by the expansion $\eqref{rho}$, actually
\[
\partial_\nu \zeta_\delta(x)=O\left(\frac{\delta^\frac{4-n}{2}}{(1+\frac{|x\mp e_n|}{\delta})^{n-3}}\right)\text{ in } \partial B_1(0)\cap B_1(\pm e_n),\qquad  \partial_\nu \zeta_\delta(x)=O(\delta^\frac{n-2}{2}) \text{ in } B_1(0)\cap B_1^c(\pm e_n).
\]
In particular, 
\begin{align}\label{zest}
\partial_\nu \zeta_\delta(x)=O\left(\frac{\delta^\frac{4-n}{2}}{(1+\frac{|x- e_n|}{\delta})^{n-3}}+  \frac{\delta^\frac{4-n}{2}}{(1+\frac{|x+ e_n|}{\delta})^{n-3}}\right)\text{ in } \partial B_1(0).    
\end{align}

We claim that 
\begin{equation}\label{eq:inner_estimates}
\zeta_\delta(x)=O\left(\frac{\delta^\frac{4-n}{2}}{(1+\frac{|x- e_n|}{\delta})^{n-3}}+  \frac{\delta^\frac{4-n}{2}}{(1+\frac{|x+ e_n|}{\delta})^{n-3}}\right) \text{ in } B_1(0),
\end{equation}
 which, combined with the first estimate in \eqref{A6} for $\varphi_0$, yields the same estimate for $|PW_{\delta}(x)-W_{\delta}(x)|$. This way, \eqref{eq:estimate_PW-W} in Lemma \ref{exp:lem} is proved.

To show \eqref{eq:inner_estimates}, let $\Omega_\delta:=\delta^{-1}(e_n-B_1)$ and let $\hat \zeta_\delta(x) := \delta^{\frac{n-2}{2}}\zeta_\delta(e_n-\delta x)$ for $x\in \Omega_\delta$. By \eqref{zest}, we have that 
\begin{equation}
\partial_\nu \hat \zeta_\delta(x)=O\left(\frac{\delta^2}{(1+|x-2\delta^{-1}e_n|)^{n-3}}+  \frac{\delta^2}{(1+|x|)^{n-3}}\right)\text{ in } \partial \Omega_\delta.
\end{equation}

If $G_\delta$ is the Green's function for the rescaled ball $\Omega_\delta:=\delta^{-1}(e_n-B_1)$, then
\begin{align*}
    G_\delta(x,y)=G(e_n-\delta x,e_n-\delta y)\delta^{n-2}
    \qquad \text{ for }x\in \Omega_\delta,\ y\in \partial \Omega_\delta,\ y\neq x.
\end{align*}
In the following we use $C$ to denote possibly different positive constants independent of $\delta$.

By \eqref{Ceq}, there is $C>0$ such that, for $x\in \Omega_\delta$,
\begin{align*}
\hat \zeta_\delta(x)=\int_{\partial \Omega_\delta} G_\delta(x,y)\partial_\nu \hat \zeta_\delta(y)\, dy
\leq C\int_{\partial \Omega_\delta}|x-y|^{2-n}\partial_\nu\hat \zeta_\delta(y)\, dy.
\end{align*}
Let
\begin{align*}
U_+&=\partial \Omega_\delta\cap B_\frac{1}{\delta}(0),\qquad
U_-=\partial \Omega_\delta\cap B_\frac{1}{\delta}\left(\frac{2e_n}{\delta}\right),\qquad
U_0=\partial \Omega_\delta\backslash (\omega_+\cup \omega_-),
\end{align*}
then $\partial \Omega_\delta = U_+\cup U_- \cup U_0.$  We estimate each of these subdomains separately. 

Note that
\begin{align}
\int_{U_0}|x-y|^{2-n}\partial_\nu\hat \zeta_\delta(y)\, dy
&\leq C\delta^{n-1}\int_{U_0}|x-y|^{2-n}\, dy
=C\int_{\delta U_0}|x-\delta^{-1}z|^{2-n}\, dz\notag\\
&\leq C\delta^{n-2} \sup_{w\in B}\int_{\partial B}|w-z|^{2-n}\, dz=O(\delta^{n-2}).\label{w0}
\end{align}
Moreover, if $x\in U_0$, then $2\delta^{-1}>|x|$, which implies that $\frac{1}{1+|x|}>\frac{\delta}{2+\delta}$, therefore,
\begin{align*}
 \frac{\delta}{(1+|x|)^{n-3}}>\delta\frac{\delta^{n-3}}{(2+\delta)^{n-3}}=\frac{\delta^{n-2}}{(2+\delta)^{n-3}},
\end{align*}
and thus
\begin{align*}
    \delta^{n-2}\leq  C\left(\frac{\delta}{(1+|x-2\delta^{-1}e_n|)^{n-3}}+ \frac{\delta}{(1+|x|)^{n-3}}\right)\qquad \text{ for } x\in U_0.
\end{align*}

Next we do the estimate for $U_+$ (the estimate for $U_-$ is analogous and also follows by symmetry). For $x\in \Omega_\delta$,
\begin{align}
&\int_{U_+}|x-y|^{2-n}\partial_\nu \hat\zeta_\delta(y)\, dy
\leq C\int_{U_+}|x-y|^{2-n}\frac{\delta^2}{(1+|y|)^{n-3}}\, dy
= C\int_{U_+}|x-y|^{2-n}\frac{\delta^2}{(1+|y|)^{1+(n-4)}}\, dy.\label{w1}
\end{align}
Then, by Lemma \ref{WY:lem},
\begin{align}
\int_{U_+}|x-y|^{2-n}\frac{\delta^2}{(1+|y|)^{1+(n-4)}}\, dy
\leq \frac{C\delta^2(\log \delta)^\tau}{(1+|x|)^{n-4}},\label{w2}
\end{align}
where $\tau=1$ if $n=4$, and $\tau=0$ if $n\geq 5$.
By \eqref{w0}, \eqref{w1}, and \eqref{w2}, we obtain that $\hat\zeta_\delta=O(\delta^2 (\log \delta)^\tau)$, which implies $\zeta_\delta=O(\delta^\frac{6-n}{2}(\log \delta)^\tau)$. On the other hand, since $\delta<\frac{2}{1+|y|}$ for $y\in U_+$, again by Lemma \ref{WY:lem} (since $n\geq 4$) we have
\begin{align}
\int_{U_+}|x-y|^{2-n}\frac{\delta^2}{(1+|y|)^{1+(n-4)}}\, dy
\leq C\int_{U_+}|x-y|^{2-n}\frac{\delta}{(1+|y|)^{1+(n-3)}}\, dy\leq  \frac{C\delta}{(1+|x|)^{n-3}},\label{w3}
\end{align}

Estimates \eqref{w0}, \eqref{w1}, and \eqref{w3} yield \eqref{eq:inner_estimates}. For the results regarding $\partial_\delta PW_\delta$, the proofs follow precisely the previous reasoning, based this time on the boundary estimates \eqref{eq:asymptotics_aux2} from Lemma \ref{lemma:asymptotics_aux}.
\end{proof}

\begin{rmk}\label{rmk:phi0_generalOmega}
If, instead of $B$, we consider a general bounded smooth domain ${\Omega}$ with the symmetries \eqref{eq:assumption_Omega}, then we may assume without loss of generality that  $e_n=(1,0,\ldots, 0)\in \partial \Omega$ is a point of positive mean curvature. We let $\varphi_0$ be the solution of
\begin{align}\label{eq:phi0_generalOmega}
-\Delta \varphi_0 = 0 \text{ in }\R^n_+,\qquad
  \frac{\partial \varphi_{0}}{\partial x_n}
  =\alpha_n\frac{n-2}{2}\sum_{j=1}^{n-1}\frac{\rho_j x_j^2}{(1+|x'|^2)^\frac{n}{2}} \text{ on }\partial \R^n_+,\qquad
 \varphi_0\to 0  \text{ as }|x|\to \infty.
\end{align}
This function satisfies  \eqref{A6}, and the expansion of Lemma \ref{exp:lem} holds true, with a similar proof. See \cite[Lemma A.1]{RW05} for a similar expansion related to positive solutions for a similar Neumann problem with a potential. Using this expansion, it is possible to adapt the proofs of Theorems \ref{thm:main:ball} and \ref{thm:main:sub} to the case of other symmetric domains. Some small changes are highlighted in Remarks \ref{rmk:generaldomain} and \ref{rmk:last}.
\end{rmk}

\section{Asymptotic estimates for \texorpdfstring{$L^q$}{} norms of bubbles}

In this appendix we collect several expansion for $L^q$ norms of the bubbles $U_{\delta,\xi}$ in different domains. We start with the following auxiliary statement.

\begin{lemma}\label{lemma:TaylorInequalities} For every $q>1$ there exists $C>0$ such that
\[
\left||a+b|^q-|a|^q\right|\leq C (|a|^{q-1}|b|+|b|^q)\qquad \forall a,b\in \R.
\]
Moreover, for $g(t):=|t|^{q-1}t$ with $q\geq 1$, there exists $C>0$ such that
\[
|g(a+b)-g(a)|\leq C (|a|^{q-1}b+|b|^q)\qquad \forall a,b\in \R
\]
while, for all $a,b\in\R$,
\[
|g(a+b)-g(a)-g'(a)b|\leq \begin{cases}
C (|a|^{q-2}b^2+|b|^q) & \text{ if } q\geq 2,\\
C|b|^q & \text{ if } 1<q<2.
\end{cases}
\]
Finally, for every $q>1$ and $\gamma\in (0,1)$ there exists $C>0$ such that
\begin{equation}\label{eq:last_estimate}
\left||a+b|^{q}\log|a+b|-|a|^{q}\log|a|\right|\leq C (|a|^{q-1-\gamma}|b|+|a|^{q-1+\gamma}|b|+|b|^{q-\gamma}+|b|^{q+\gamma})\qquad \forall a,b\in \R.
\end{equation}
\end{lemma}
\begin{proof}
 The proof of the first three statements follows directly from a Taylor expansion with Lagrange remainder. As for the fourth (the case $1<q<2$), it is equivalent to proving that the function
 \[
 h(x):=||x+1|^{q-1}(x+1)-|x|^{q-1}x-q|x|^{q-1}| 
 \]
 is bounded in $\R$. Since $h$ is continuous in $\R$, one just needs to check it is bounded at infinity. For $x>0$ large, we have, by Taylor expansion, for some $\xi\in(x,x+1)$ that
 \[
 h(x)=q(q-1)|\xi|^{q-2}\leq \frac{q(q-1)}{|x|^{2-q}}\to 0 \quad \text{ as } x\to \infty, \]
 and similarly of $x<0$ large, hence the claim follows.
 Finally, for \eqref{eq:last_estimate} we take $i(x)=|a|^{q}\log|a|$ and have:
 \[
 i(a+b)-i(a)=i'(s)b,\quad \text{ for some } s \text{ between $a$ and $b$}.
 \]
 Then
 \begin{align*}
 |i'(s)b &|\leq |s|^{q-2}(|s|+|s\log |s||) |b|\leq c |s|^{q-2}(|s|+|s|^{1-\gamma}+|s|^{1+\gamma}) |b|\\
 			&\leq c (|s|^{q-1-\gamma}+|s|^{p-1+\gamma})|b|\leq C(|a|^{q-1-\gamma}|b|+|a|^{q-1+\gamma}|b|+|b|^{q-\gamma}+|b|^{q+\gamma})
 \end{align*}
\end{proof}
For $\eps>0$ and $t\in \R$, let 
\begin{align*}
f_\eps(t):=|t|^{p-1+\eps}t. 
\end{align*}

\begin{lemma}\label{lemma:f_eps_expansion} For $t\in \R$,
\begin{align}\label{eq:f_eps(t)_expansion}
f_\eps(t)=|t|^{p-1}t+ \eps |t|^{p-1}t \log |t|+ \eps^2 r_{1,\eps}(t),
\end{align}
and
\begin{align*}
f'_\eps(t) = p|t|^{p-1}+ \eps (|t|^{p-1}+p|t|^{p-1} \log |t|)+ \eps^2 r_{2,\eps}(t),
\end{align*}
where 
$|r_{1,\eps}(t)|\leq  \frac{1}{2}\left(|t|^p+|t|^{p+ \eps}\right) (\log |t|)^2$
and 
$|r_{2,\eps}(t)|\leq  2(p+1)\left(|t|^{p-1}+|t|^{p-1+\eps}\right)(\log|t|+ (\log |t|)^2)$.
\end{lemma}
\begin{proof}
For every $t\in \R$, we make the expansion around $\eps=0$: there exists $\sigma\in (0,\eps)$ such that
\begin{align*}
f_\eps(t) &=f_0(t)+ \eps\frac{\partial}{\partial \eps} f_\eps(t)|_{\eps=0}  + \frac{\eps^2}{2}\frac{\partial^2}{\partial \eps^2} f_\eps(t) |_{\eps=\sigma}\\
	      &=|t|^{p-1}t+  \eps |t|^{p-1} t \log |t| + \eps^2 t|t|^{p-1+\sigma}(\log |t|)^2,
\end{align*}
and use the estimate $|t|^{p+\sigma}\leq |t|^p+|t|^{\eps}$. The expansion for $f'_\eps(t)$ is similar.
\end{proof}

\begin{lemma}\label{lemma:LpnormBubble} Let $\alpha_n:=[n(n-2)]^\frac{n-2}{4}$. For $\xi\in \R^n$ and $\delta>0$, recall that
\[
U_{\delta,\xi}(x)=\alpha_n \frac{\delta^{\frac{n-2}{2}}}{(\delta^2+|x-\xi|^2)^{\frac{n-2}{2}}},\quad \text{ and define } \quad V_{\delta,\xi}(x):=\frac{\delta^\frac{n-2}{2}}{(\delta+|x-\xi|)^{n-3}}.
\]For $R>0$, as $\delta\to 0^+$ we have
\[
\int_{B_R(\xi)} U_{\delta,\xi}^q=\begin{cases}
O(\delta^{q\frac{n-2}{2}}) & \text{ if } 0<q<\frac{n}{n-2},\\
O(\delta^\frac{n}{2}|\log \delta|) & \text{ if } q=\frac{n}{n-2},\\
O(\delta^{n-q\frac{n-2}{2}}) & \text{ if } \frac{n}{n-2}<q<\infty,\\
\end{cases}
\]
and
\[
\int_{B_R(\xi)} V_{\delta,\xi}^q=\begin{cases}
O(\delta^{q\frac{n-2}{2}}) & \text{ if } 0<q<\frac{n}{n-3},\\
O(\delta^\frac{n(n-2)}{2(n-3)}|\log \delta|) & \text{ if } q=\frac{n}{n-3},\\
O(\delta^{n-q\frac{n-4}{2}}) & \text{ if } \frac{n}{n-3}<q<\infty.
\end{cases}
\]
In particular, if $\xi\in \overline{{B}}$, $|U_{\delta,\xi}|_{2^*+O(\delta)}=O(1)$ as $\delta\to 0$.
\end{lemma}
\begin{proof}
The proof of the statement for $U_{\delta,\xi}$ can be found for instance in \cite[Lemma A.3]{PistoiaSoave18}. As for $V_{\delta,\xi}$, by making the change of variables $x=\xi+\delta y$, we see that
\[
\int_{B_R(\xi)} V_{\delta,\xi}^q=\delta^{n-q\frac{n-4}{2}}\int_{B_\frac{R}{\delta}(0)} \frac{1}{(1+|y|)^{q(n-3)}}\, dy.
\]
If $q>n/(n-3)$, the last integral is bounded as $\delta\to 0$, hence $\int_{B_R(\xi)} V_{\delta,\xi}^q=O(\delta^{n-q\frac{n-4}{2}})$. In case $q\leq n/(n-3)$, by using generalized polar coordinates,
\begin{align*}
\int_{B_R(\xi)} V_{\delta,\xi}^q \leq  C\delta^{n-q\frac{n-4}{2}}\left(1+\int_1^\frac{R}{\delta} r^{n-1-q(n-3)}\, dr \right)= \begin{cases} O(\delta^{q\frac{n-2}{2}}) & \text{ if } q<\frac{n}{n-3}\\
O(\delta^\frac{n(n-2)}{2(n-3)}|\log \delta|) & \text{ if } q=\frac{n}{n-3}
\end{cases}
\end{align*}
\end{proof}

Recall that ${B}:=B_1(0)$, and denote 
\begin{align*}
\omega_+:= {B}\cap B_\frac{1}{2}(e_n)\qquad \text{ and }\qquad 
\omega_-:= {B}\cap B_\frac{1}{2}(-e_n).    
\end{align*}
In particular, 
\begin{align}\label{bd}
\delta^2+|x\mp e_n|^2\geq \delta^2+(\tfrac{1}{2})^2\geq \tfrac{1}{4}\qquad \text{for $x\in {B}\backslash \omega_\pm$}.
\end{align}
Namely,
\begin{align*}
|U_{\delta,\pm e_n}|\leq \alpha_n \frac{\delta^{\frac{n-2}{2}}}{(\delta^2+|x\mp e_n|^2)^{\frac{n-2}{2}}}
\leq \alpha_n 2^{n-2}\delta^{\frac{n-2}{2}}\qquad \text{for $x\in {B}\backslash \omega_\pm$}.
\end{align*}

We follow the ideas in \cite{AdimurthiMancini} to prove the following result. Since we use a different notation and for the ball the computations are explicit, we include the proof for completeness.

\begin{lemma}\label{AB:lem}
If $n\geq 4$, then
\begin{align*}
\int_{B} U_{\delta,e_n}^{p+1}=\int_{B} U_{\delta,-e_n}^{p+1}=
\frac{\mathfrak{A}}{2}-\frac{\mathfrak{B}}{2}\delta +o(\delta)
\qquad \text{as }\delta\to 0^+,
\end{align*}
where $\mathfrak{A}$ and $\mathfrak{B}$ are given in \eqref{fraks}.
\end{lemma}
\begin{proof}
Note that
\begin{align}\label{eq:Lp+1norm_aux1}
\int_{\omega_+} U_{\delta,e_n}^{p+1}= \underbrace{\frac{1}{2}\int_{B_\frac{1}{2}(e_n)} U_{\delta,e_n}^{p+1}}_{(I)}- \underbrace{\int_{\Sigma} U_{\delta,e_n}^{p+1}}_{(II)},
\end{align}
where 
\begin{align*}
\Sigma:=B^-_\frac{1}{2}(e_n)\setminus \omega_+,\qquad B^-_\frac{1}{2}(e_n)=\{y\in B_\frac{1}{2}(e_n)\::\: y_n<1\}. 
\end{align*}
 Let $\Delta:=\{x'\in \R^{n-1}:\ |x'|<\sqrt{15}/8\}$ be the projection of the set $\partial {B} \cap B_\frac{1}{2}(e_n)$ in the variables $x':=(x_1,\ldots, x_{n-1})$, so that $\Sigma=\{x'\in \Delta:\ \sqrt{1-|x'|^2}<x_n<1\}$.

On the one hand, we have
\begin{align*}
(I)=\frac{1}{2}\int_{\R^n} U_{\delta,e_n}^{p+1} - \frac{1}{2}\int_{\R^n\setminus B_\frac{1}{2}(e_n)}U_{\delta,e_n}^{p+1}=\frac{1}{2}\int_{\R^n} \frac{\alpha_n^{p+1}}{(1+|x|^2)^n}\, dx+O(\delta^n)=\frac{\mathfrak{A}}{2}+O(\delta^n).
\end{align*}
Likewise,
\begin{align*}
(II)=\int_{\Delta}\int_{\sqrt{1-|x'|^2}}^1\frac{\alpha_n^{p+1}\delta^n}{(\delta^2+|x-e_n|^2)^n}\, dx_ndx'.
\end{align*}
Using now the change of variables $y_n=\frac{1-x_n}{\sqrt{\delta^2+|x'|^2}}$, so that $\delta^2+|x-e_n|^2=(\delta^2+|x'|^2)(1+y_n^2)$, we see that
\[
(II)=\int_{\Delta} \frac{\alpha_n^{p+1}\delta^n }{(\delta^2+|x'|^2)^{n-\frac{1}{2}}}  \int_0^{\frac{1-\sqrt{1-|x'|^2}}{\sqrt{\delta^2+|x'|^2}}} \frac{1}{(1+y_n^2)^n}\, dy_ndx'.
\]
Observe that
\[
\int_0^s \frac{1}{(1+y_n^2)^n}\, dy_n=s+O(s^3)\quad  \text{ and }\quad  1-\sqrt{1-|x'|^2}=\frac{|x'|^2}{2} + O(|x'|^4)
\]
uniformly for $s\in \R$ and for $x'\in \Delta$, respectively. So
\begin{align*}
(II)&=\int_{\Delta} \frac{\alpha_n^{p+1}\delta^n}{(\delta^2+|x'|^2)^{n-\frac{1}{2}}}\left(\frac{1-\sqrt{1-|x'|^2}}{\sqrt{\delta^2+|x'|^2}} + O\left(\frac{1-\sqrt{1-|x'|^2}}{\sqrt{\delta^2+|x'|^2}}\right)^3\right)\, dx'\\
	&=\int_{\Delta} \frac{\alpha_n^{p+1}\delta^n}{(\delta^2+|x'|^2)^{n}} \left(\frac{|x'|^2}{2} + O(|x'|^4) \right) \,  dx'\\
	& \underset{(x'=\delta y')}{=} \int_{\Delta/\delta} \frac{\alpha_n^{p+1}\delta^{-1}}{(1+|y'|^2)^n}\left(\frac{\delta^2|y'|^2}{2}+O(\delta^4|y'|^4)\right)\, dy'\\
	&= \delta \frac{\alpha_n^{p+1}}{2} \int_{\Delta/\delta} \frac{|y'|^2}{(1+|y'|^2)^n}\, dy' + O(\delta^3) \int_{\Delta/\delta} \frac{|y'|^4}{(1+|y'|^2)^n}\, dy' \\
	&= \delta \frac{\alpha_n^{p+1}}{2} \int_{\R^{n-1}} \frac{|y'|^2}{(1+|y'|^2)^n}\, dy'  +  \delta \frac{\alpha_n^{p+1}}{2} \int_{\R^{n-1}\setminus (\Delta/\delta)} \frac{|y'|^2}{(1+|y'|^2)^n}\, dy'+ o(\delta)\\
	&= \delta \frac{\alpha_n^{p+1}}{2} \int_{\R^{n-1}} \frac{|y'|^2}{(1+|y'|^2)^n}\, dy' +o(\delta)=\frac{\mathfrak{B}}{2}\delta+o(\delta),
\end{align*}
where $O(\delta^3) \int_{\Delta/\delta} \frac{|y'|^4}{(1+|y'|^2)^n}\, dy'=o(\delta)$ since 
\begin{align*}
\int_{\Delta/\delta} \frac{|y'|^4}{(1+|y'|^2)^n}\, dy'\to \int_{\R^{n-1}}\frac{|y'|^4}{(1+|y'|^2)^n}\, dy' <\infty \quad \text{ for } n\geq 4.
\end{align*}

The claim now follows, since
\begin{align*}
\int_{{B}\backslash \omega_+} U_{\delta,e_n}^{p+1}
= \int_{{B}\backslash \omega_+} \frac{\delta^{n}}{(\delta^2+|x+e_n|^2)^{n}}
\leq |{B}| 4^{n}\delta^{n}.
\end{align*}
The proof for $U_{\delta,-e_n}$ follows by symmetry. 
\end{proof}

\begin{rmk}\label{rmk:generaldomain}
If, instead of $B$, we consider a general smooth bounded domain ${\Omega}$ and $\xi_0\in \partial {\Omega}$, then we would have the expansion
\[
\int_{\Omega} U_{\delta,\xi_0}^{p+1}=
\frac{\mathfrak{A}}{2}-H(\xi_0)\frac{\mathfrak{B}}{2}\delta +o(\delta)
\qquad \text{as }\delta\to 0^+,
\]
where $H(\xi_0)$ is the mean curvature of $\partial {\Omega}$ at $\xi_0$. The proof can be found in \cite{AdimurthiMancini}: up to a rotation and translation, we can assume that $\xi_0=0$ and that, for sufficiently small $R>0$,
\[
{\Omega}\cap B_R=\{(x',x_n):\ x_n>\rho(x')\},\quad \rho(x')=\sum_{j=1}^{n-1} \rho_j x_j^2+O(|x'|^3),\quad H(\xi_0)=\frac{2\sum_{j=1}^{n-1}\rho_j}{n-1}.
\]
With respect to the proof of the previous lemma, in the general case the quantity $(I)$ has the same expansion, while
\begin{align*}
(II)&=\int_\Delta \int_0^{\varphi(x')} \frac{\alpha_n^{p+1}\delta^n}{(\delta^2+|x-e_n|^2)^n}\, dx_n\, dx'=\int_\Delta \frac{\alpha_n^{p+1}\delta^n}{(\delta^2+|x'|^2)^n}\left(\sum_{j=1}^{n-1} \rho_j x_j^2+O(|x'|^3)\right)\, dx'\\
        &=\delta \alpha_n^{p+1} \sum_{j=1}^{n-1}\rho_j \int_\Delta \frac{|x'|^2}{(\delta^2+|x'|^2)^n}\, dx'+o(\delta)=H(0)\frac{\mathfrak{B}}{2}\delta+o(\delta).
\end{align*}
\end{rmk}

Next, we focus on the interaction between bubbles centered at different points. 
\begin{lemma}\label{inter:lem}
We have
\begin{align*}
    \int_{B} U_{\delta,e_n}U_{\delta,-e_n}^p
    =\int_{B} U_{\delta,-e_n}U_{\delta,e_n}^p
    =O(\delta^{n-2})\qquad \text{ as }\delta\to 0^+.
\end{align*}
\end{lemma}
\begin{proof}
Note that
\begin{align*}
 \int_{B} U_{\delta,e_n}U_{\delta,-e_n}^p
& = \alpha_n^{p+1}\int_{B} \frac{\delta^{\frac{n-2}{2}}}{(\delta^2+|x-e_n|^2)^{\frac{n-2}{2}}}
 \left(\frac{\delta^{\frac{n-2}{2}}}{(\delta^2+|x+e_n|^2)^{\frac{n-2}{2}}}\right)^\frac{n+2}{n-2}\, dx\\
&= \alpha_n^{p+1}\int_{B} \frac{\delta^{n}}{(\delta^2+|x-e_n|^2)^{\frac{n-2}{2}}(\delta^2+|x+e_n|^2)^{\frac{n+2}{2}}}\, dx.
\end{align*}
Let $D:=\omega_+-e_n$. Then
\begin{align*}
&\int_{\omega_+} \frac{\delta^{n}}{(\delta^2+|x-e_n|^2)^{\frac{n-2}{2}}(\delta^2+|x+e_n|^2)^{\frac{n+2}{2}}}\, dx
\leq
2^{n+2}
\int_{D} \frac{\delta^{n}}{(\delta^2+|x|^2)^{\frac{n-2}{2}}}\, dx \\
& \leq  2^{n+2}
\int_{D/\delta} \frac{\delta^{n+2}}{(1+|y|^2)^{\frac{n-2}{2}}}\, dy \leq C 2^{n+2}
\delta^{n+2}\left(1
+
\int_{1}^\frac{1}{2\delta} \frac{r^{n-1}}{(1+r^2)^{\frac{n-2}{2}}}\ dr
\right)\\
&\leq C 2^{n+2}
\delta^{n+2}
\left(1
+
\int_{1}^\frac{1}{2\delta}r\ dr
\right)
=C 2^{n+2}
\delta^{n+2}\left(\frac{1}{2}
+\frac{1}{8\delta^2}
\right)=O(\delta^n).
\end{align*}
Reasoning in the same way,
\begin{align*}
\int_{\omega_-} \frac{\delta^{n}}{(\delta^2+|x-e_n|^2)^{\frac{n-2}{2}}(\delta^2+|x+e_n|^2)^{\frac{n+2}{2}}}\, dx\leq C2^{n-2}\delta^{n-2}\left(1+\int_1^{\frac{1}{2\delta}} \frac{r^{n-1}}{(1+r^2)^\frac{n+2}{2}}\, dr\right)=O(\delta^{n-2}),
\end{align*}
and
\[
\int_{{B} \backslash (\omega_+\cup \omega_-)} \frac{\delta^{n}}{(\delta^2+|x-e_n|^2)^{\frac{n-2}{2}}(\delta^2+|x+e_n|^2)^{\frac{n+2}{2}}}\, dx
\leq 
|{B}|4^n \delta^{n}=O(\delta^n). \qedhere
\]

\end{proof}

Recall that, by \eqref{phinull} and \eqref{A6}, for $x\in\R^n_+$,
\begin{align}\label{phiest}
\varphi_0(x)&:=\frac{\alpha_n}{\omega_{n}}\int_{\R^{n-1}}\frac{|y|^2}{(1+|y|^2)^\frac{n}{2}|x-y|^{n-2}}\ dy,\qquad |\varphi_0(x)|\leq \frac{C}{(1+|x|)^{n-3}}.
\end{align}

\begin{lemma}\label{I2:1} We have
\begin{align*}
\delta^{-\frac{n-4}{2}}\int_{B}\varphi_0\Big(\frac{e_n-x}{\delta}\Big)U_{\delta,-e_n}^p=\delta^{-\frac{n-4}{2}}\int_{B}\varphi_0\Big(\frac{e_n+x}{\delta}\Big)U_{\delta,e_n}^p
=O(\delta^{n-2})\qquad \text{ as }\delta\to 0.
\end{align*}
\end{lemma}
\begin{proof}
Let $D:=\omega_+-e_n$. Using \eqref{bd} and \eqref{phiest},and making the change of variable $x\mapsto x-e_n$:
\begin{align*}
    &\delta^{-\frac{n-4}{2}}\int_{\omega_+}\varphi_0\Big(\frac{e_n-x}{\delta}\Big)U_{\delta,-e_n}^p
    =\delta^{-\frac{n-4}{2}}\alpha_n^p\int_{\omega_+}\varphi_0\Big(\frac{e_n-x}{\delta}\Big)
    \frac{\delta^{\frac{n+2}{2}}}{(\delta^2+|x+e_n|^2)^{\frac{n+2}{2}}}\\
    &\leq \alpha_n^p2^{n+2}\delta^{3}\int_{D}\varphi_0\Big(\frac{-x}{\delta}\Big)
    \leq \alpha_n^p2^{n+2}\delta^{n+3}\int_{\frac{1}{\delta}D}\frac{C}{(1+|x|)^{n-3}}\\
    &\leq C\alpha_n^p2^{n+2}\delta^{n+3}\omega_n
    \left(1
    +\int_{1}^\frac{1}{2\delta}r^{2}
    \right)
    =O(\delta^{n}),
\end{align*}
where we recall that $\omega_n$ denotes the measure of the unit sphere of $\R^n$. On the other hand, if $K:=\omega_-+e_n$, 
\begin{align*}
&\delta^{-\frac{n-4}{2}}\int_{\omega_-}\varphi_0\Big(\frac{e_n-x}{\delta}\Big)U_{\delta,-e_n}^p
=\delta^{1-n}\alpha_n^p\int_{K}\varphi_0\Big(\frac{2e_n-x}{\delta}\Big)
    \frac{1}{(1+|\frac{x}{\delta}|^2)^{\frac{n+2}{2}}}\\
&=\delta\alpha_n^p\int_{\frac{1}{\delta}K}\varphi_0\Big(\frac{2e_n}{\delta}-x\Big)\frac{1}{(1+|x|^2)^{\frac{n+2}{2}}}
\leq C \delta^{n-2}\alpha_n^p\left(\frac{2}{3}\right)^{n-3}\int_{\frac{1}{\delta}K}\frac{1}{(1+|x|^2)^{\frac{n+2}{2}}}
=O(\delta^{n-2}),
\end{align*}
since  $|2e_n\delta^{-1}-x|\geq 2\delta^{-1}-(2\delta)^{-1}=\frac{3}{2}\delta^{-1}$ for $x\in \delta^{-1}K$ and therefore, by \eqref{phiest},
\begin{align*}
|\varphi_0\left(\frac{2e_n}{\delta}-x\right)|\leq \frac{C}{(1+|2e_n\delta^{-1}-x|)^{n-3}}\leq C \left(\frac{2}{3}\right)^{n-3}\delta^{n-3},\qquad x\in \delta^{-1}K.
\end{align*}
Finally,
\begin{align*}
\delta^{-\frac{n-4}{2}}\int_{{B}\backslash(\omega_+\cup \omega_-)}\varphi_0\Big(\frac{e_n-x}{\delta}\Big)U_{\delta,-e_n}^p
\leq 
\left(\delta^{-\frac{n-4}{2}}\right)
\left(C2^{n-3}\delta^{n-3}\right)
\left(|{B}| 2^{n+2}\delta^{\frac{n+2}{2}}\alpha_n^p\right)=O(\delta^n),
\end{align*}
because, for $x\in {B}\backslash(\omega_+\cup \omega_-)$, we have  $|x\pm e_n|\geq \frac{1}{2}$ and
\begin{align}\label{phiest1}
\left|\varphi_0\left(\frac{e_n-x}{\delta}\right)\right|&\leq \frac{C}{(1+\frac{1}{2}\delta^{-1})^{n-3}}
\leq C2^{n-3}\delta^{n-3},\quad 
U_{\delta,-e_n}^p= \left(\frac{\delta}{\delta^2+|x+e_n|^2}\right)^\frac{n+2}{2}
\leq 2^{n+2} \delta^\frac{n+2}{2}
\end{align}
\end{proof}

\begin{lemma}\label{I2:2} As $\delta\to 0^+$,
\begin{align*}
\delta^{-\frac{n-4}{2}}\int_{B}\varphi_0\Big(\frac{e_n-x}{\delta}\Big)U_{\delta,e_n}^p=\delta^{-\frac{n-4}{2}}\int_{B}\varphi_0\Big(\frac{e_n+x}{\delta}\Big)U_{\delta,-e_n}^p
&=-\frac{n-2}{2}\alpha_n\delta\int_{\R^{n-1}}
 \frac{|y|^2}{(1+|y|^2)^{n-1}}+o(\delta)\\
 &=-\frac{n-2}{2}\mathfrak{C}\delta+o(\delta),
\end{align*}
where we recall that $\mathfrak{C}$ was introduced in \eqref{fraks}.
\end{lemma}
\begin{proof}
Let $D:=\omega_+-e_n$ and observe that, after a change of variables,
\begin{align}
  &\delta^{-\frac{n-4}{2}}\int_{\omega_+}\varphi_0\Big(\frac{e_n-x}{\delta}\Big)U_{\delta,e_n}^p   =\alpha_n^p\delta\int_{\frac{1}{\delta}D}
 \frac{\varphi_0(-x)}{(1+|x|^2)^{\frac{n+2}{2}}}
 =\alpha_n^p\delta\int_{\R^n_+}
 \frac{\varphi_0(x)}{(1+|x|^2)^{\frac{n+2}{2}}}+o(\delta)\notag\\
 &=\delta\int_{\R^n_+}\varphi_0(x)(-\Delta U_{1,0})+o(\delta)=\delta\int_{\partial\R^{n}_+}(\partial_\nu\varphi_0)U_{1,0}+o(\delta)\notag\\
 &=-\delta\int_{\R^{n-1}}
 \alpha_n\frac{n-2}{2}\frac{|y|^2}{(1+|y|^2)^\frac{n}{2}(1+|y|^2)^\frac{n-2}{2}}+o(\delta)
 =-\frac{n-2}{2}\mathfrak{C}\delta+o(\delta),\label{1}
\end{align}
where we have used the decay of $\varphi_0,U_{1,0}$ to integrate by parts, together with the fact that $\partial_\nu U_{1,0}=0$ on $\partial \R^n_+$.
On the other hand, using that $|x-e_n|\geq \frac{1}{2}$  for $x\in {B}\backslash\omega_+$,
we have that \eqref{phiest1} holds true and that $U_{\delta,e_n}\leq 2^{n+2}\delta^\frac{n+2}{2}$. Therefore,
\begin{align}
\delta^{-\frac{n-4}{2}}\int_{{B}\backslash\omega_+}\varphi_0\Big(\frac{e_n-x}{\delta}\Big)U_{\delta,e_n}^p  
\leq \left(C2^{n-3}\delta^{n-3}\right)\left(2^{n+2} \delta^\frac{n+2}{2}\right)=O(\delta^{n+\frac{n-4}{2}})=o(\delta).\label{2}
\end{align}
The final claim now follows by \eqref{1}, \eqref{2}, and by symmetry of the integrand functions.
\end{proof}

\begin{rmk}\label{rmk:last}
If, instead of $B$, we consider a general smooth bounded domain ${\Omega}$ with the symmetries \eqref{eq:assumption_Omega}, then we may assume without loss of generality that  $e_n=(1,0,\ldots, 0)\in \partial {\Omega}$ is a point of positive mean curvature.    Then similar versions of Lemmas \ref{I2:1} and \ref{I2:2} hold true for $\varphi_0$ as in \eqref{eq:phi0_generalOmega} (recall Remark \ref{rmk:phi0_generalOmega}), with the only difference that in the expansion of the second lemma one has $-\frac{n-2}{2}H(e_n)\mathfrak{C}\delta+o(\delta)$.
\end{rmk}

Next we collect some asymptotic estimates for integral normal of 
\[
W_\delta:=U_{\delta,e_n}-U_{\delta,-e_n} \quad \text{ and } \quad w_\delta:=P W_{\delta},
\]which for convenience we recall to be the solution of
\begin{align*}
-\Delta P W_{\delta}=-\Delta  U_{\delta,e_n}+\Delta  U_{\delta,-e_n}\ \hbox{in}\ {B},\qquad \partial_\nu P W_{\delta}=0\ \hbox{on}\ \partial{B}, \qquad \int\limits_{{B}} P W_{\delta}dx=0.
\end{align*}

\begin{lemma}\label{delta:lem}
We have, as $\delta\to 0^+$,
\begin{align*}
\left\|\delta \partial_\delta w_\delta\right\|^2=\left(\frac{n-2}{2}\right)^2\int_{\R^n}\frac{\alpha_n^{p+1}(|x|^2-1)^2}{(1+|x|^2)^{n+2}}\, dx + o(1).
\end{align*}
\end{lemma}
\begin{proof}
Integrating by parts,
\begin{align*}
 \left\|\delta \partial_\delta w_\delta\right \|^2&=-\delta^2 \int_{B} (\Delta \partial_\delta w_\delta) \partial_\delta w_\delta=
 \int_{B} \delta \partial_\delta\left(U_{\delta,e_n}^p-U_{\delta,-e_n}^p\right) \delta \partial_\delta w_\delta\\
 &=p\int_{B} (U_{\delta,e_n}^{p-1}(\delta \partial_\delta U_{\delta,e_n})^2 + U_{\delta,-e_n}^{p-1}(\delta \partial_\delta U_{\delta,-e_n})^2) + \psi(\delta),
\end{align*}
where
\begin{align*}
\psi(\delta)&=-2p \int_{B} (U_{\delta,e_n}^{p-1}+U_{\delta,e_n}^{p-1})(\delta \partial_\delta U_{\delta,e_n})(\delta \partial_\delta U_{\delta,-e_n})\\
&+p\int_{B} (U_{\delta,e_n}^{p-1}(\delta \partial_\delta U_{\delta,e_n})-U_{\delta,-e_n}^{p-1}(\delta \partial_\delta U_{\delta,-e_n}))(\delta \partial_\delta(w_\delta-W_\delta)).
\end{align*}
Observe that, by Lemma \ref{exp:lem} we have
\[
| \delta \partial_\delta (w_\delta-W_\delta)|\leq C(V_{\delta,e_n}+V_{\delta,-e_n}),
\]
where $V_{\delta,\xi}$ are as in Lemma \ref{lemma:LpnormBubble}. Taking in consideration \eqref{pa U con U}, we see that
\[
|\psi(\delta)|\leq C \int_{B} (U_{\delta,e_n}^{p-1}+U_{\delta,-e_n}^{p-1})(U_{\delta,e_n}U_{\delta,-e_n}+V_{\delta,e_n}+V_{\delta,-e_n})=o(1),
\]
by Lemma \ref{lemma:LpnormBubble}. Moreover,
\begin{align*}
\int_{B} U_{\delta,e_n}^{p-1}(\delta \partial_\delta U_{\delta,e_n})^2&= \left(\frac{n-2}{2}\right)^2\int_{B} \frac{\alpha_n^{p+1}\delta^n(|x-e_n|^2-\delta^2)^2}{(\delta^2+|x-e_n|^2))^{n+2}}\, dx= \left(\frac{n-2}{2}\right)^2\int_{\frac{{B}-e_n}{\delta}} \frac{\alpha_n^{p+1}(|y|^2-1)^2}{(1+|y|^2)^{n+2}}\, dy\\
&=\left(\frac{n-2}{2}\right)^2\int_{\R^n_+}\frac{\alpha_n^{p+1}(|y|^2-1)^2}{(1+|y|^2)^{n+2}}\, dy + o(1).
\end{align*}
The conclusion for the term 
\[
\int_{B} U_{\delta,-e_n}^{p-1}(\delta \partial_\delta U_{\delta,-e_n})^2
\]
and the conclusion of the lemma follows by symmetry.
\end{proof}

\begin{lemma}\label{Lemma:I_1aux}
For $\gamma\geq 0$  small we have, as $\delta\to 0^+$:
\[
|f_0(w_\delta) - f_0(W_\delta)|_{\frac{(p+1)(1+\gamma)}{p}}=\begin{cases}
O(\delta|\log \delta|^\frac{1}{4}) & \text{ if } n=4,\ \gamma=0,\\
O(\delta) & \text{ if } n\geq 5,\ \gamma=0,\\
O(\delta^{1-\frac{(n+2)\gamma}{2(1+\gamma)}}) & \text{ if } n\geq 4,\ \gamma>0,
\end{cases}
\]
and
\[
|f_0(W_\delta) - f_0(U_{\delta,e_n})+f_0(U_{\delta,-e_n})|_{\frac{(p+1)(1+\gamma)}{p}}=\begin{cases}
O(\delta^{n-2}) & \text{ if } 4\leq n<6,\ \gamma\geq 0,\\
O(\delta^4|\log \delta|^\frac{2}{3}) & \text{ if } n=6,\ \gamma=0,\\
O(\delta^{\frac{n+2}{2}}) & \text{ if } n=6,\ \gamma>0 \text{ or } n>6,\ \gamma\geq 0.
\end{cases}
\]
\end{lemma}
\begin{proof}

\noindent a) Estimate of $ |f_0(w_\delta) - f_0(W_\delta)|_{\frac{(p+1)(1+\gamma)}{p}}$ for $\gamma\geq 0$. Using the second statement of Lemma \ref{lemma:TaylorInequalities} with $q:=p$, $a:=W_\delta$ and $b:=w_\delta-W_\delta$, followed by H\"older's inequality,
\begin{align*}
\int_{B} |f_0&(w_\delta)-f_0(W_\delta)|^\frac{(p+1)(1+\gamma)}{p} = \int_{B} \left||w_\delta|^{p-1}W
_\delta- |W_\delta|^{p-1}W_\delta \right|^\frac{(p+1)(1+\gamma)}{p}\\
		& \leq C \int_{B} \left(|W_\delta|^{p-1}|w_\delta-W_\delta| + |w_\delta -W_\delta|^p \right)^\frac{(p+1)(1+\gamma)}{p}	\\
		& \leq C' \int_{B} |W_\delta|^\frac{(p-1)(p+1)(1+\gamma)}{p}|w_\delta-W_\delta|^\frac{(p+1)(1+\gamma)}{p}+C' \int_{B} |w_\delta-W_\delta|^{(p+1)(1+\gamma)}\\
		& \leq C'' \left(\int_{B} |W_\delta|^{(p+1)(1+\gamma)}\right)^\frac{p-1}{p}\left(\int_{B} |w_\delta-W_\delta|^{(p+1)(1+\gamma)}\right)^\frac{1}{p}+C' \int_{B} |w_\delta-W_\delta|^{(p+1)(1+\gamma)}.
\end{align*}
Now, recalling that $p+1=\frac{2n}{n-2}$ and using Lemma \ref{lemma:LpnormBubble}:
\[
\int_{B} |W_\delta|^{(p+1)(1+\gamma)}\leq  C\int_{B} \left(|U_{\delta,e_n}|^{(p+1)(1+\gamma)}+|U_{\delta,-e_n}|^{(p+1)(1+\gamma)} \right) =\begin{cases}
O(1) & \text{ if } \gamma=0,\\
O(\delta^{-{n\gamma}}) & \text{ if } \gamma>0.
\end{cases}
\]
Moreover, a consequence of Lemma \ref{exp:lem} combined with Lemma \ref{lemma:LpnormBubble}:
\[
\int_{B} |w_\delta-W_\delta|^{(p+1)(1+\gamma)}=
\begin{cases}
O(\delta^4 |\log \delta|) & \text{ if } n=4,\ \gamma=0,\\
O(\delta^{\frac{2n}{n-2}}) & \text{ if } n>4,\ \gamma=0,\\
O(\delta^\frac{2n-n\gamma(n-4)}{n-2}) & \text{ if } n\geq 4,\ \gamma>0.
\end{cases}
\]
In conclusion,
\begin{align*}
\int_{B} |f_0&(w_\delta)-f_0(W_\delta)|^\frac{(p+1)(1+\gamma)}{p}=\begin{cases}
O(\delta^{\frac{p+1}{p}}|\log \delta|^\frac{1}{3}) & \text{ if } n=4,\ \gamma=0\\
O(\delta^{\frac{p+1}{p}}) & \text{ if } n\geq 5,\ \gamma=0,\\
O\left(\delta^{\frac{p+1}{p}\left(1-\frac{n\gamma}{2}\right)}\right) & \text{ if } n\geq 4,\ \gamma>0.
\end{cases}
\end{align*}
\smallbreak

\noindent b) Estimate of $|f_0(W_\delta) - f_0(U_{\delta,e_n})+f_0(U_{\delta,-e_n})|_{\frac{(p+1)(1+\gamma)}{p}}$. We use the notation 
\begin{align*}
\omega_+:={B}\cap B_{\frac{1}{2}}(e_n)\qquad \text{ and }\qquad \omega_-:={B}\cap B_{\frac{1}{2}}(-e_n).    
\end{align*}

Observe that 
\begin{align}
\int_{\omega_+} | f_0(W_\delta) - f_0(U_{\delta,e_n})+&f_0(U_{\delta,-e_n}) |^\frac{(p+1)(1+\gamma)}{p} =\int_{\omega_+} \left| |W_\delta|^{p-1}W_\delta - U_{\delta,e_n}^p+U_{\delta,-e_n}^p \right|^\frac{(p+1)(1+\gamma)}{p} \nonumber \\
				&\leq C \int_{\omega_+} \left| |W_\delta|^{p-1}W_\delta - U_{\delta,e_n}^p \right|^\frac{(p+1)(1+\gamma)}{p} + C\int_{\omega_+} U_{\delta,-e_n}^{(p+1)(1+\gamma)}\label{eq:estimateR_aux2}.
\end{align}
Now, since $|x+e_1|\geq 1$ in $\omega_+$,
\begin{align}\label{eq:estimateR_aux3}
\int_{\omega_+} U_{\delta,-e_n}^{(p+1)(1+\gamma)} = \int_{\omega_+} \left( \frac{\alpha_n^{p+1}\delta^n}{(\delta^2+|x+e_1|^2)^n}\right)^{1+\gamma} \, dx \leq \alpha_n^{(p+1)(1+\gamma)} |\omega_+| \delta^{n(1+\gamma)} = O(\delta^{n(1+\gamma)}).
\end{align}
Applying the second statement of Lemma \ref{lemma:TaylorInequalities} with $q:=p$, $a:=U_{\delta,e_n}$ and $b:=-U_{\delta,-e_2}$, we obtain
\begin{align}\label{eq:estimateR_aux3.2}
\int_{\omega_+} \left| |W_\delta|^{p-1}W_\delta-U_{\delta,e_n}^p \right|^\frac{(p+1)(1+\gamma)}{p} 
			&\leq C \int_{\omega_+} \left(|U_{\delta,e_n}|^{p-1}U_{\delta,-e_n} + |U_{\delta,-e_n}|^{p}\right)^\frac{(p+1)(1+\gamma)}{p}\\
			& \leq C' \int_{\omega_+} U_{\delta,e_n}^\frac{(p-1)(p+1)}{p} U_{\delta,-e_n}^\frac{(p+1)(1+\gamma)}{p} + C' \int_{\omega_+} |U_{\delta,-e_n}|^{(p+1)(1+\gamma)} \nonumber \\
			&\leq O(\delta^\frac{n(n-2)(1+\gamma)}{n+2}) \int_{\omega_+} |U_{\delta,e_n}|^\frac{8n(1+\gamma)}{(n+2)(n-2)} + O(\delta^{n(1+\gamma)}). \label{eq:estimateR_aux1}
\end{align}
By Lemma \ref{lemma:LpnormBubble} with $R=2$, $q:=\frac{8n}{(n+2)(n-2)}$, and observing that $0<q<\frac{2n}{n-2}$, and that $q<\frac{n}{n-2}$ if and only if $n>6$, we have
\begin{equation}\label{eq:estimateR_aux4}
\int_{\omega_+} U_{\delta,e_n}^\frac{8n}{(n+2)(n-2)} =\begin{cases}
O(\delta^\frac{n(n-2)}{n+2}) & \text{ if } 4\leq n<6,\\
O(\delta^3|\log \delta|) & \text{ if } n=6,\\
O(\delta^\frac{4n}{n+2}) & \text{ if } n>6.
\end{cases}
\end{equation}
while, for $\gamma>0$ sufficiently small:
\begin{equation}\label{eq:estimateR_aux5}
\int_{\omega_+} U_{\delta,e_n}^\frac{8n(1+\gamma)}{(n+2)(n-2)} =\begin{cases}
O(\delta^\frac{n(n-2)(1+\gamma)}{n+2}) & \text{ if } 4\leq n<6,\\
O(\delta^\frac{4n(1+\gamma)}{n+2}) & \text{ if } n\geq 6.
\end{cases}
\end{equation}

Going back to \eqref{eq:estimateR_aux1} and combining it with \eqref{eq:estimateR_aux4}--\eqref{eq:estimateR_aux5}, we see that
\begin{align*}
\int_{\omega_+} \left| |W_\delta|^{p-1}W_\delta - U_{\delta,e_n}^p+U_{\delta,-e_n}^p \right|^\frac{(p+1)(1+\gamma)}{p} 
		&=\begin{cases}
O(\delta^\frac{2n(n-2)(1+\gamma)}{n+2}) & \text{ if } 3\leq n<6,\ \gamma\geq 0 \text{ small},\\
O(\delta^6|\log \delta|) & \text{ if } n=6,\ \gamma=0\\
O(\delta^{n(1+\gamma)}) & \text{ if } n>6,\ \gamma>0 \text{ or } n>6,\ \gamma\geq 0.
\end{cases}
\end{align*}
Exchanging the roles of $e_n$ and $-e_n$, we have precisely the same type of estimate for 
\[
\int_{\omega_-} \left| |W_\delta|^{p-1}W_\delta - U_{\delta,e_n}^p+U_{\delta,-e_n}^p \right|^\frac{(p+1)(1+\gamma)}{p}.
\] Finally, since $|x-e_n|,|x+e_n|\geq \frac{1}{2}$ for $x\in {B} \backslash (\omega_+ \cup \omega_-)$, we have
\begin{align*}
\int_{{B} \backslash (\omega_+ \cup \omega_-)} \left| |W_\delta|^{p-1}W_\delta - U_{\delta,e_n}^p+U_{\delta,-e_n}^p \right|^\frac{(p+1)(1+\gamma)}{p} 	&\leq C \int_{{B} \backslash (\omega_+ \cup \omega_-)} \left( U_{\delta,e_n}^{(p+1)(1+\gamma)}+ U_{\delta,-e_n}^{(p+1)(1+\gamma)}\right)\\
&=O(\delta^{n(1+\gamma)}),
\end{align*}
which concludes the proof.
\end{proof}

\begin{lemma}\label{lemma:|i*(f_eps(w_delta)-f_0(w_delta))|_aux}

For $\ell>0$ small, let $\eta\in (0,1)$, $d>0$ and $\delta=d\eps$. Then, for every $\gamma\in(0,1)$,
\[
|f_\eps(w_\delta)-f_0(w_\delta)|_\frac{(p+1)(1+\ell)}{p}
= O(\eps^{1-\gamma}).
\]
as $\eps\to 0$, uniformly in $d\in (1/\eta,\eta)$.
\end{lemma}
\begin{proof}
Consider the case $\ell=0$. Let $\sigma>0$. Using the expansion \eqref{eq:f_eps(t)_expansion}, 
\begin{align*}
f_\eps(t)&=f_0(t)+ \eps |t|^{p-1}t \log |t|+ \eps^2 r_\eps(t),\\ 
|r_\eps(t)|&\leq  C(|t|^p + |t|^{p+\bar \eps})(\log |t|)^2,
\end{align*}
where we take $\bar \eps$ such that $(p+\bar \eps)(p+1)/p <(2n+2\sigma)/(n-2)$. Then (recall that $(p+1)/p=2n/(n+2)$),
\begin{align*}
\int_{B} |f_\eps(w_\delta)-f_0(w_\delta)|^\frac{p+1}{p} &= \int_{B} |\eps w_\delta |w_\delta|^{p-1}\log |w_\delta| + \eps^2 r_\eps(w_\delta)|^\frac{p+1}{p}\\
			&\leq C \eps^\frac{p+1}{p}\int_{B} |w_\delta|^{p+1}(\log |w_\delta|)^\frac{p+1}{p} + C \eps^\frac{2(p+1)}{p}\int_{B} |r_\eps(w_\delta)|^\frac{p+1}{p} \\
			& \leq C''\eps^\frac{2n}{n+2} \left(1+\int_{B} |w_\delta|^\frac{2n+2\sigma}{n-2} + \eps^\frac{2n}{n+2}\int_{B} |w_\delta|^\frac{2n+2\sigma}{n-2}\right).
\end{align*}
We can now conclude by recalling that $|w_\delta|\leq (U_{\delta,e_n}+U_{\delta,-e_n})$ - see Lemma \ref{exp:lem} - and by applying Lemma \ref{lemma:LpnormBubble}, which yields
\[
\int_{B} |w_\delta|^\frac{2n+2\sigma}{n-2} \leq C \int_{B} \left(U_{\delta,e_n}^\frac{2n+2\sigma}{n-2}+U_{\delta,-e_n}^\frac{2n+2\sigma}{n-2} \right)=O(\delta^{-\sigma}). \qedhere
\]
The case $\ell>0$ small is analogous, see the proof of Lemma \ref{Lemma:I_1aux}.
\end{proof}

\begin{lemma}\label{lemmaA}
Let $d>0$ and $\delta=d\eps$. Then, for every $\gamma\in(0,1)$,
\[
|f'_0(w_\delta)|_{\frac{n}{2}}=O(1),\qquad |f'_\eps(w_\delta)-f'_0(w_\delta)|_{\frac{n}{2}}\leq o(\eps^{1-\gamma}).
\]
\end{lemma}
\begin{proof}
The first estimate follows from Corollary \ref{cor:estimates_with_bubbles} and Lemma \ref{lemma:LpnormBubble}.  For the second estimate, note that, by Lemma \ref{lemma:f_eps_expansion}, 
\begin{align*}
f'_\eps(t)-f'_0(t) &=\eps (|t|^{p-1}+p|t|^{p-1} \log |t|)+ \eps^2 r_{2,\eps}(t),\\
|r_{2,\eps}(t)|&\leq  2(p+1)\left(|t|^{p-1}+|t|^{p-1+\eps}\right)(\log|t|+ (\log |t|)^2)
\end{align*}
and therefore,
\begin{align*}
&\int_{B} |f'_\eps(w_\delta)-f'_0(w_\delta)|^\frac{n}{2}
=\int_{{B}} \left| |w_{\delta}|^{\frac{4}{n-2}+\eps}-|w_\delta|^{\frac{4}{n-2}} \right|^\frac{n}{2}\\
&\leq C \eps\left(\int_{{B}}|w_\delta|^{p+1}+
|w_\delta|^{p+1}|\log|w_\delta||^{\frac{n}{2}}+\eps^\frac{n}{2}(|w_\delta|^{\frac{4}{n-2}}+|w_\delta|^{\frac{4}{n-2}+\eps})^\frac{n}{2}(|\log|w_\delta||+|\log|w_\delta||^2)^{\frac{n}{2}}
\,dx\right).
\end{align*}
Now the claim follows arguing as in Lemma \ref{lemma:|i*(f_eps(w_delta)-f_0(w_delta))|_aux}.
\end{proof}

\bibliographystyle{plain}

\bigskip

\begin{flushleft}
\textbf{Angela Pistoia}\\
Dipartimento di Scienze di Base e Applicate per l’Ingegneria\\
Sapienza Universita di Roma\\
Via Scarpa 16, 00161 Roma, Italy\\
\texttt{angela.pistoia@uniroma1.it} 
\vspace{.3cm}

\textbf{Alberto Saldaña}\\
Instituto de Matemáticas\\
Universidad Nacional Autónoma de México\\
Circuito Exterior, Ciudad Universitaria\\
04510 Coyoacán, Ciudad de México, Mexico\\
\texttt{alberto.saldana@im.unam.mx} 
\vspace{.3cm}

\textbf{Hugo Tavares}\\
Departamento de Matemática do Instituto Superior Técnico\\
Universidade de Lisboa\\
Av. Rovisco Pais\\
1049-001 Lisboa, Portugal\\
\texttt{hugo.n.tavares@tecnico.ulisboa.pt} 
\end{flushleft}

\end{document}